\documentclass[12pt,reqno, basicdelimiters]{NumPDEsArticle}
\usepackage[utf8]{inputenc}
\usepackage[T1]{fontenc}
\usepackage[english]{babel}

\usepackage{NumPDEsMacros}

\usepackage{csquotes}
\usepackage{enumitem}
\usepackage{amsmath,amssymb}
\usepackage{microtype}
\usepackage{bbm}
\usepackage{bm}
\usepackage{tikz}
\usepackage{subfig}
\usepackage{graphicx}
\usepackage{stackengine}
\usepackage{xcolor}
\usepackage{float}
\usepackage{orcidlink}
\usepackage{soul}

\newcommand{\osc}{\mathrm{osc}}

\makeatletter
\DeclareBibliographyDriver{article}{%
  \usebibmacro{bibindex}%
  \usebibmacro{begentry}%
  \usebibmacro{author/translator+others}%
  \setunit{\printdelim{nametitledelim}}\newblock
  \usebibmacro{title}%
  \newunit
  \printlist{language}%
  \newunit\newblock
  \usebibmacro{byauthor}%
  \newunit\newblock
  \usebibmacro{bytranslator+others}%
  \newunit\newblock
  \printfield{version}%
  \newunit\newblock
  \usebibmacro{journal+issuetitle}%
  \newunit
  \usebibmacro{byeditor+others}%
  \newunit\newblock
  \iftoggle{bbx:isbn}
    {\printfield{issn}}
    {}%
  \newunit\newblock
  \usebibmacro{addendum+pubstate}%
  \newunit\newblock
  \printfield{note}%
  \newunit\newblock
  \usebibmacro{doi+eprint+url}%
  \setunit{\bibpagerefpunct}\newblock
  \usebibmacro{pageref}%
  \newunit\newblock
  \iftoggle{bbx:related}
    {\usebibmacro{related:init}%
     \usebibmacro{related}}
    {}%
  \usebibmacro{finentry}}
\makeatother

\title{Optimal convergence of adaptive BEM driven by functional-type error estimators}

\author{Maximilian Brunner\,\orcidlink{0000-0003-0636-1491}}
\author{Alexander Freiszlinger\,\orcidlink{0009-0005-0877-8137}}
\author{Dirk Pauly\,\orcidlink{0000-0003-4155-7297}}
\author{Dirk Praetorius\,\orcidlink{0000-0002-1977-9830}}

\address{TU Wien, Institute of Analysis and Scientific Computing, Wiedner Hauptstraße 8-10, 1040 Wien, Austria}

\email{maximilian.brunner@asc.tuwien.ac.at}
\email{alexander.freiszlinger@asc.tuwien.ac.at \quad \rm (corresponding author)}

\address{TU Dresden, Institute of Analysis, Zellescher Weg 12-14, 01069 Dresden, Germany}
\email{dirk.pauly@tu-dresden.de}

\address{TU Wien, Institute of Analysis and Scientific Computing, Wiedner Hauptstraße 8-10, 1040 Wien, Austria}

\email{dirk.praetorius@asc.tuwien.ac.at}

\setul{0.5ex}{0.3ex}
\setulcolor{blue}

\thanks{This research was funded by the Austrian Science Fund (FWF) projects
\href{https://www.fwf.ac.at/en/research-radar/10.55776/F65}{10.55776/F65} (SFB
F65 ``Taming complexity in PDE systems''),
\href{https://www.fwf.ac.at/en/research-radar/10.55776/I6802}{10.55776/I6802}
(international project I6802 ``Functional error estimates for PDEs on unbounded
domains''), and
\href{https://www.fwf.ac.at/en/research-radar/10.55776/P33216}{10.55776/P33216}
(standalone project P33216 ``Computational nonlinear PDEs'').}

\begin{document}

\maketitle 

\begin{abstract}
In the present work, we derive functional upper bounds for the potential error arising from boundary element discretizations of the Laplace--Dirichlet problem.
These bounds are based on local auxiliary problems on patches of boundary vertices and the resulting \textsl{a~posteriori} error estimator is shown to be locally equivalent to the well-studied residual error estimator.
This equivalence result allows us to prove $R$-linear convergence of the functional \textsl{a~posteriori} error estimator and, together with a suitable mesh-refining strategy, to establish that the potential error as well as the functional error estimator converge with optimal rates with respect to the number of boundary elements.
Numerical experiments affirm the theoretical findings and illustrate the practical performance of the related adaptive algorithm driven by the proposed functional error estimator.
\end{abstract}

\section{Introduction}

Let $\Omega \subset \R^d$ with $d \in \set{2,3}$ be a Lipschitz domain with compact boundary $\Gamma \coloneqq \partial \Omega$.
We consider the Laplace--Dirichlet problem of finding the potential $u^\star$ such that
\begin{equation} \label{eq:Laplace}
 - \Delta u^\star = 0 \quad \text{in } \Omega 
 \quad \text{ subject to } \quad 
 u^\star = g \quad \text{on } \Gamma.
\end{equation}
We employ a boundary integral approach and consider an indirect ansatz using the single-layer potential, i.e.,
\begin{equation} \label{eq:ansatz}
 u^\star = \widetilde{V} \phi^\star \coloneqq \int_\Gamma G(\cdot,y) \phi^\star(y) \d y,
\end{equation}
where $\phi^\star$ is an unknown density and $G$ is the fundamental solution of the Laplace equation, i.e.,
\begin{equation} \label{eq:fundamental}
 G(x,y) \coloneqq
 \begin{cases}
  -\frac{1}{2\pi} \log \abs{x - y} & \text{if } d = 2, \\
  \frac{1}{4\pi} \abs{x - y}^{-1} & \text{if } d = 3.
 \end{cases} 
\end{equation}
Restricting~\eqref{eq:ansatz} to the boundary $\Gamma$, we obtain the weakly-singular boundary integral equation of the first kind
\begin{equation} \label{eq:BIE}
 V \phi^\star \coloneqq (\widetilde{V} \phi^\star)|_\Gamma = g,
\end{equation}
where the integral representation of $V$ on the boundary coincides with that of $\widetilde{V}$ in the domain; see, e.g.,~\cite[Section~7]{McLean2000}.
The boundary integral operator $V \colon H^{-1/2}(\Gamma) \to H^{1/2}(\Gamma)$ is bounded, symmetric and, provided that $\mathrm{diam}(\Gamma) < 1$ for $d = 2$, also elliptic.
Hence, the Lax--Milgram lemma ensures that there exists a unique solution $\phi^\star \in H^{-1/2}(\Gamma)$ to~\eqref{eq:BIE} and, consequently, a unique solution $u^\star$ to~\eqref{eq:Laplace} given by~\eqref{eq:ansatz}.

Given a mesh $\TT_H^\Gamma$ of $\Gamma$ and a polynomial degree $p \in \N_0$, the boundary element method (BEM) seeks to discretize~\eqref{eq:BIE} and to approximate $\phi^\star$ by a $\TT_H^\Gamma$-piecewise polynomial $\phi_H \in \PP^p(\TT_H^\Gamma)$ on the boundary $\Gamma$.
The resulting approximation of the potential $u^\star$ is then given by
\begin{equation} \label{eq:potapprox}
  u_H \coloneqq \widetilde{V} \phi_H \approx u^\star.
\end{equation}

Independently of the specific discretization, the fundamental properties of $\widetilde{V}$ ensure that the approximate potential $u_H$ and hence the potential error $u^\star - u_H$ are harmonic, i.e., $\Delta u_H = 0 = \Delta (u^\star - u_H)$ in $\Omega$.
Based on the harmonicity of $u^\star - u_H$ and the Dirichlet principle, we derive a functional \textsl{a~posteriori} error estimator $\eta_H$ that provides an essentially constant-free (up to data oscillations) upper bound for the potential error, i.e., 
\begin{equation*} \label{eq:potbound}
 \norm{\nabla(u^\star - u_H)}_\Omega \leq \eta_H + \const{C}{osc} \osc_H, 
\end{equation*}
and, in particular, empirically
\begin{equation*}
 \norm{\nabla(u^\star - u_H)}_\Omega \approx \eta_H,
\end{equation*} 
where $\const{C}{osc} > 0$ denotes a generic constant.
Functional error estimation strategies have successfully been applied to various settings and problem formulations, such as elliptic PDEs in bounded domains~\cite{Ainsworth1997,Braess2008,Repin2008,Ern2015,Anjam2016} and unbounded domains~\cite{Pauly2009}, the boundary element method for the Laplace--Dirichlet problem~\cite{Kurz2021,Freiszlinger2025}, and the full-space transmission problem~\cite{Freiszlinger2026}.

While similar error estimation strategies rely on solving a discrete auxiliary problem on a boundary strip $\omega_H \subseteq \Omega$ (see, e.g.,~\cite{Kurz2021,Freiszlinger2025}), the proposed functional upper bound in this paper is based on local auxiliary problems on patches of boundary vertices. 
With this new approach, this work improves the error estimation strategy proposed in~\cite{Kurz2021} in the sense that the resulting functional error estimator is able to more accurately identify areas, where the error is comparatively large while retaining the properties of the error estimator from~\cite{Kurz2021}, i.e., guaranteed upper bounds for the potential error independently of the discretization method. 
Beyond these improvements, a major contribution of this work lies in proving optimal convergence rates with respect to the number of boundary elements for the potential error arising from Galerkin BEM, i.e., $\phi_H = \phi_H^\star \in \PP^p(\TT_H^\Gamma)$ is the Galerkin approximation to~\eqref{eq:BIE} in $\PP^p(\TT_H^\Gamma)$ and $u_H^\star \coloneqq \widetilde{V} \phi_H^\star \approx u^\star$ is the reconstructed approximation$\colon$ For a constant $\const{C}{opt} > 0$ that is independent of the number of boundary elements, there holds
\begin{equation} \label{eq:introopt}
 \norm{\nabla(u^\star - u_H^\star)}_\Omega
 \leq \const{C}{opt} \bigl(\#\TT_H^\Gamma\bigr)^{-s},
\end{equation}
where the decay rate $s > 0$ is guaranteed to be as large as possible and hence optimal with respect to the usual nonlinear approximation classes.

While existing results on optimal convergence of Galerkin BEM are restricted to the residual-based error estimator~\cite{Carstensen1997,Carstensen2001} and its properties (see, e.g.,~\cite{Feischl2013,Gantumur2013,Feischl2014,Gantner2022,Gantner2022a}), we note that certain properties like \emph{stability} and \emph{reduction} (see~\cite{Carstensen2014}) are not known to hold for the functional error estimator proposed in the present work.
Instead, we reveal certain local equivalences between the new functional error estimator and the residual error estimator.
These equivalences allow us to prove $R$-linear convergence of the functional error estimator and, with the usual mesh refinement by newest-vertex bisection, to establish that the potential error as well as the functional error estimator converge with optimal rates with respect to the number of boundary elements.

\textbf{Outline.} The paper is structured as follows.
In Section~\ref{sec:prelim}, we collect the necessary notation and useful preliminary results about function spaces, meshes, and Galerkin BEM.
Section~\ref{sec:apost} derives the new functional estimates based on local auxiliary problems on patches of boundary vertices (Section~\ref{sec:funcest}), proposes an associated adaptive algorithm (Section~\ref{sec:adap}), and states the main results (Section~\ref{sec:mainresults}), i.e., $R$-linear convergence (Theorem~\ref{thm:linconv}) and optimal convergence rates (Theorem~\ref{thm:opt}).
In Section~\ref{sec:equiv}, we prove the core result that the proposed functional error estimator and the residual error estimator are locally equivalent (Theorem~\ref{thm:equiv}). 
The proof of $R$-linear convergence of the functional error estimator is given in Section~\ref{sec:linconv}. 
In Section~\ref{sec:optconv}, we prove the claim of Theorem~\ref{thm:opt} that the potential error and the functional error estimator converge with optimal rates with respect to the number of boundary elements. 
Finally, Section~\ref{sec:numerics} concludes the work with some numerical experiments which underline the theoretical findings and illustrate the performance of the adaptive algorithm. 

\section{Preliminaries} \label{sec:prelim}

\subsection{General notation}

Let $d \in \set{2,3}$. Let $\Omega \subset \R^d$ be a Lipschitz domain with compact boundary $\Gamma \coloneqq \partial \Omega$. 
By $\abs{\cdot}$, we denote, without any ambiguity, the absolute value of a scalar, the Euclidean norm of a vector in $\R^n$ for any $n \in \N$, the $d$-dimensional Lebesgue measure of a measurable set in $\R^d$, or the $(d-1)$-dimensional Hausdorff measure of a surface (piece).
We write $\int_U \cdot \d x$ for integration over a Lebesgue measurable set $U \subseteq \overline{\Omega}$ or over a Hausdorff measurable set $U \subseteq \Gamma$.

Throughout, discrete objects or quantities associated with a discrete object are indicated by an index, e.g., $\TT_H$, $\phi_H$ etc.
Moreover, we always use the same index for related objects, e.g., $\phi_H^\star$ is the Galerkin BEM solution associated with the mesh $\TT_H^\Gamma$ and $\rho_H$ is the residual error estimator associated with the mesh $\TT_H$. 

Lastly, we abbreviate notation in proofs and write $a \lesssim b$ if there exists a generic constant $C > 0$ such that $a \leq C b$.
The precise dependencies of the hidden constant are either clear from the context or explicitly stated in the text.
If $a \lesssim b$ and $b \lesssim a$, we write $a \simeq b$.

\subsection{Function spaces}

For Lipschitz domains $U \subseteq \R^d$, we denote by $L^2(U)$ the usual space of square-integrable functions on $U$ with corresponding inner product resp.\ norm
\begin{equation*}
 \dual{v}{w}_U \coloneqq \int_U v w \d x,
 \quad \norm{v}_U \coloneqq \dual{v}{v}_U^{1/2}.
\end{equation*}
For vector-valued spaces $[L^2(U)]^n$ with $n \in \N$, the inner product and norm are defined componentwise and we will omit $n$ if it is clear from the context.
For measurable boundary pieces $S \subseteq \Gamma$, the space $L^2(S)$, the scalar product $\dual{\cdot}{\cdot}_S$, and the norm $\norm{\cdot}_S$ are defined analogously with respect to the $(d-1)$-dimensional Hausdorff measure.

We define the usual first-order Sobolev space
\begin{equation*}
  H^1(U) \coloneqq \set{v \in L^2(U) \text{ weakly differentiable } \given \nabla v \in L^2(U)}
\end{equation*}
with associated norm
\begin{equation*}
 \norm{v}_{H^1(U)} \coloneqq \bigl( \norm{v}_U^2 + \norm{\nabla v}_U^2 \bigr)^{1/2},
\end{equation*}
where $\nabla$ denotes the distributional gradient.
For unbounded domains $U$ with compact boundary $\partial U$, we introduce the weight function $\varrho \colon U \to \R_{>0}$ defined by
\begin{equation*}
 \varrho(x) \coloneqq
 \begin{cases}
  (1 + \abs{x})^{-1} (1 + \log(1 + \abs{x}))^{-1} & \text{if } d = 2, \\
  (1 + \abs{x})^{-1} & \text{if } d = 3
 \end{cases}
\end{equation*}
and define the weighted $L^2$-space
\begin{equation*}
 L^2_{\varrho}(U) \coloneqq \set{v \colon U \to \R \text{ measurable } \given \varrho v \in L^2(U)}.
\end{equation*}
Additionally, we introduce the weighted Sobolev space
\begin{equation*}
 H^1_{\varrho}(U) \coloneqq \set{v \in L^2_{\varrho}(U) \text{ weakly differentiable } \given \nabla v \in L^2(U)}.
\end{equation*}
Since this work deals with the Laplace--Dirichlet problem~\eqref{eq:Laplace} on bounded and unbounded domains, we use the following convention to shorten the presentation$\colon$ We write $L^2_\varrho(U)$ and $H^1_\varrho(U)$ throughout, with the understanding that $\varrho \equiv 1$ if $U$ is bounded.

For boundary pieces $S \subseteq \Gamma$ and with the tangential gradient $\nabla_\Gamma$, we define
\begin{equation*}
 H^1(S) \coloneqq \set{\psi \in L^2(S) \text{ weakly differentiable } \given \nabla_\Gamma \psi \in L^2(S)}
\end{equation*}
with associated norm
\begin{equation*}
 \norm{\psi}_{H^1(S)} \coloneqq \bigl( \norm{\psi}_S^2 + \norm{\nabla_\Gamma \psi}_S^2 \bigr)^{1/2}.
\end{equation*}
Furthermore, we define the fractional Sobolev space 
\begin{equation*}
  H^{1/2}(S) \coloneqq \Big \{ \psi \in L^2(S) \, | \, \abs{\psi}_{H^{1/2}(S)}^2 \coloneqq \int_S \int_S \frac{\abs{\psi(x) - \psi(y)}^2}{\abs{x - y}^d} \d x \d y < \infty \Big \},
\end{equation*}
which is associated with the norm
\begin{equation*}
 \norm{\psi}_{H^{1/2}(S)} \coloneqq \bigl( \norm{\psi}_S^2 + \abs{\psi}_{H^{1/2}(S)}^2 \bigr)^{1/2}.
\end{equation*}
For Lipschitz domains $U$ (bounded or unbounded) with compact boundary $\partial U$, it is well-known that $H^{1/2}(\partial U)$ is the trace space of $H^1_{\varrho}(U)$, i.e., there holds 
\begin{equation} \label{eq:traceineq}
  H^{1/2}(\partial U) = \set{v|_{\partial U} \given v \in H^1_{\varrho}(U)} 
  \quad \text{ and } \quad  
  \norm{v|_{\partial U}}_{H^{1/2}(\partial U)} \leq \const{C}{trace}\norm{v}_{H^1_{\varrho}(U)},
\end{equation}
where $\const{C}{trace} > 0$ depends only on $U$.
Lastly, we define the negative-order Sobolev space $H^{-1/2}(S)$ as the dual of $H^{1/2}(S)$, where the duality pairing is defined by continuous extension of the $L^2(S)$-inner product, i.e.,
\begin{equation*}
  \dual{\phi}{\psi}_{H^{-1/2}(S) \times H^{1/2}(S)} \coloneqq \dual{\phi}{\psi}_S
  \quad \text{ for all } \phi \in L^2(S) \text{ and } \psi \in H^{1/2}(S).
\end{equation*}
To ease readability, we will write $\dual{\cdot}{\cdot}_S$ instead of $\dual{\cdot}{\cdot}_{H^{-1/2}(S) \times H^{1/2}(S)}$.

\subsection{Meshes and newest-vertex bisection} \label{subsec:meshes}

Throughout this work, each mesh $\TT_H$ of $U \subset \R^d$ will be a conforming triangulation of the bounded set $U$ by compact simplices $T \in \TT_H$, i.e., there holds $\overline{U} = \bigcup_{T \in \TT_H} T$ and the intersection of two different simplices $T,T' \in \TT_H$ is either empty, a common vertex, a common edge, or a common face (for $d = 3$). 

Throughout, we assume that all considered meshes are obtained from some fixed initial mesh $\TT_0$ by finitely many steps of newest-vertex bisection (NVB). 
We refer to~\cite{Stevenson2008} for NVB for $d \geq 2$ with additional assumptions on the initial mesh $\TT_0$, as well as to \cite{Karkulik2012} (for $d = 2$) and \cite{Diening2025} (for general $d \geq 2$) for the case of arbitrary initial meshes. 
Given a set of marked elements $\MM_H \subseteq \TT_H$, we denote by
\begin{equation*}
 \TT_h \coloneqq \mathtt{refine}(\TT_H, \MM_H)
\end{equation*}
the coarsest mesh where all elements in $\MM_H$ are refined by at least one bisection, i.e., $\MM_H \subseteq \TT_H \setminus \TT_h$.
For some given mesh $\TT_H$, let $\T(\TT_H)$ be the set of all meshes that can be obtained from $\TT_H$ by finitely many steps of newest-vertex bisection.
For $\TT_H = \TT_0$, we simply write $\T \coloneqq \T(\TT_0)$.

Clearly, any mesh $\TT_H$ on $U$ induces a conforming boundary mesh 
\begin{equation*}
  \TT_H^{\partial U} \coloneqq \TT_H|_{\partial U} \coloneqq \set{T \cap \partial U \given T \in \TT_H \text{ and } \abs{T \cap \partial U} > 0}.
\end{equation*}
We note that NVB guarantees uniform $\kappa$-shape regularity of all meshes in $\T$ and their respective boundary meshes, i.e., there exists $\kappa > 0$  depending only on $\TT_0$ such that
\begin{equation} \label{eq:shapereg}
  \sup_{T \in \TT_H} \frac{\mathrm{diam}(T)}{\abs{T}^{1/d}} 
  + \sup_{F \in \TT_H^{\partial U}} \frac{\mathrm{diam}(F)}{\abs{F}^{1/(d-1)}} 
  + \sup_{\substack{F,F' \in \TT_H^{\partial U} \\ F \cap F' \neq \emptyset}} \frac{\mathrm{diam}(F)}{\mathrm{diam}(F')}
  \leq \kappa 
  \quad \text{for all } \TT_H \in \T.
\end{equation}
\subsection{Vertices and patches}

For each mesh $\TT_H$, let $\NN_H$ denote the set of all vertices $z$ of $\TT_H$ and $\NN_H^{\partial U} \coloneqq \NN_H \cap \partial U$ be the set of all vertices of $\TT_H^{\partial U}$.
Given $k \in \N$ and some set $V \subseteq U$, we define the patch and the $k$-patch of $V$ as
\begin{equation*}
  \TT_H[V] \coloneqq \set{T \in \TT_H \given V \cap T \neq \emptyset}
  \quad \text{ and } \quad
 \TT_H^k[V] \coloneqq
 \begin{cases} 
  \TT_H[V] & \text{if } k = 1, \\
  \TT_H\big [\bigcup\TT_H^{k-1}[V]\big] & \text{if } k > 1.
 \end{cases}
\end{equation*}
We set $U_H[V] \coloneqq \bigcup_{T \in \TT_H[V]} T$.
If $V = \set{z}$ is a single vertex, we simply write $\TT_H[z] \coloneqq \TT_H[\set{z}]$ and $U_H[z] \coloneqq U_H[\set{z}]$.
If $V = \bigcup \UU_H$ for some subset $\UU_H \subseteq \TT_H$, we write $\TT_H[\UU_H] \coloneqq \TT_H[\bigcup \UU_H]$ and $U_H[\UU_H] \coloneqq U_H[\bigcup \UU_H]$.

Patches of boundary pieces $S \subseteq \partial U$ with respect to a boundary mesh $\TT_H^{\partial U}$ are defined accordingly, i.e., $\TT_H^{\partial U}[S] \coloneqq \set{F \in \TT_H^{\partial U} \given S \cap F \neq \emptyset}$ and $\partial U_H[S] \coloneqq \bigcup \TT_H^{\partial U}[S]$.

Finally, we define the mesh-size function $h \colon \bigcup_{\TT_H \in \T}\TT_H \to \R_{>0}$ by $h_T \coloneqq \abs{T}^{1/d}$.
For boundary facets $F \in \bigcup_{\TT_H \in \T} \TT_H^\Gamma$, we write $h_F \coloneqq \abs{F}^{1/(d-1)}$.

\subsection{Discrete spaces} \label{subsec:discspace}

For $r \in \N_0$ and $U \subseteq \R^d$, we define $\P^r(U)$ as the space of polynomials of total degree at most $r$ on $U$.
Given a mesh $\TT_H$ of some bounded Lipschitz domain $U \subseteq \R^d$, we define the spaces of (in general discontinuous) piecewise polynomials
\begin{equation*}
 \begin{split}
  \PP^q(\TT_H) &\coloneqq \set{v \colon U \to \R \given v|_T \in \P^q(T) \text{ for all } T \in \TT_H}, \\
  \PP^q(\TT_H^{\partial U}) &\coloneqq \set{v \colon \partial U \to \R \given v|_F \in \P^q(F) \text{ for all } F \in \TT_H^{\partial U}},
 \end{split}
\end{equation*}
and the usual $H^1$-conforming finite element spaces
\begin{equation*}
 \begin{split}
  \SS^p(\TT_H) = \PP^p(\TT_H) \cap C(\overline{U}) \subset H^1(U)
  \ \text{ and } \ 
  \SS^p(\TT_H^{\partial U}) = \PP^p(\TT_H^{\partial U}) \cap C(\partial U) \subset H^1(\partial U)
 \end{split}
\end{equation*}
as well as 
\begin{equation*}
 \SS^p_0(\TT_H) \coloneqq \set{v \in \SS^p(\TT_H) \given v|_{\partial U} = 0}.
\end{equation*}
We introduce the hat functions $(\zeta_{H,z})_{z \in \NN_H} \subset \SS^1(\TT_H)$ defined by 
\begin{equation} \label{eq:hatdef}
 \zeta_{H,z}(z') \coloneqq
  \begin{cases}
    1 & \text{if } z' = z, \\
    0 & \text{if } z' \in \NN_H \setminus \set{z}.
  \end{cases}
\end{equation}
It is well-known that the hat functions form a basis of $\SS^1(\TT_H)$ and constitute a partition of unity on $\overline{U}$ as well as on its boundary $\partial U$, i.e.,
\begin{equation} \label{eq:hatunity}
 \sum_{z \in \NN_H} \zeta_{H,z} = 1 \quad \text{pointwise on } \overline{U}
 \qquad \text{ and } \qquad  \sum_{z \in \NN_H^{\partial U}} \zeta_{H,z}|_{\partial U} = 1 \quad \text{pointwise on } \partial U.
\end{equation}
Lastly, we define the $L^2(\Gamma)$-orthogonal projection $Q_H \colon L^2(\Gamma) \to \PP^p(\TT_H^\Gamma)$, which satisfies
\begin{equation} \label{eq:L2proj}
  \norm{(1 - Q_H)\psi}_F \leq \const{C}{app} h_F \norm{\nabla_\Gamma \psi}_F
  \quad \text{for all } F \in \TT_H^\Gamma \text{ and } \psi \in H^1(\Gamma),
\end{equation}
where $\const{C}{app} > 0$ depends only on $\kappa$-shape regularity of $\TT_H^\Gamma$.
\subsection{Galerkin BEM}

We consider the single-layer operator $\widetilde{V}$
\begin{equation*}
  \widetilde{V} \colon H^{-1/2}(\Gamma) \to H^1_\varrho(\Omega) \quad \text{defined by} \quad (\widetilde{V} \phi)(x) \coloneqq \int_\Gamma G(x,y) \phi(y) \d y
  \quad \text{for all } x \in \Omega,
\end{equation*}
where $G$ is the fundamental solution of the Laplace equation; see~\eqref{eq:fundamental}.
Setting $V \coloneqq \widetilde{V}|_\Gamma \colon H^{-1/2}(\Gamma) \to H^{1/2}(\Gamma)$, we have that the integral representation of $V$ coincides with that of $\widetilde{V}$ on the boundary.
Additionally, $V$ is a bounded isomorphism and, provided that $\mathrm{diam}(\Gamma) < 1$ for $d = 2$, elliptic, i.e., there exists $\const{C}{ell} > 0$ depending only on $\Gamma$ such that
\begin{equation*}
  \enorm{\phi}^2 \coloneqq \dual{V\phi}{\phi}_\Gamma \geq \const{C}{ell} \norm{\phi}_{H^{-1/2}(\Gamma)}^2
  \quad \text{for all } \phi \in H^{-1/2}(\Gamma);
\end{equation*} 
see, e.g.,~\cite[Section~7]{McLean2000}. Hence, given $g \in H^{1/2}(\Gamma)$, the Lax--Milgram lemma ensures existence and uniqueness of the solution $\phi^\star \in H^{-1/2}(\Gamma)$ to~\eqref{eq:BIE} and $\norm{\phi^\star}_{H^{-1/2}(\Gamma)} \leq \const{C}{ell}^{-1} \norm{g}_{H^{1/2}(\Gamma)}$.

Given $p \in \N_0$, we consider the discrete problem of finding $\phi_H^\star \in \PP^p(\TT_H^\Gamma)$ such that
\begin{equation} \label{eq:discprob}
 \dual{V \phi_H^\star}{\psi_H}_\Gamma = \dual{g}{\psi_H}_\Gamma
 \quad \text{for all } \psi_H \in \PP^p(\TT_H^\Gamma).
\end{equation}
By virtue of the Lax--Milgram lemma, \eqref{eq:discprob} admits a unique solution $\phi_H^\star \in \PP^p(\TT_H^\Gamma)$.
An approximation $u_H^\star$ of $u^\star$ is then extracted by $u_H^\star \coloneqq \widetilde{V} \phi_H^\star$.

\section{Adaptive algorithm and main result} \label{sec:apost}

Throughout the remainder of this work, suppose that $g \in H^1(\Gamma)$ and let $\omega \subseteq \Omega$ be a bounded subset such that $\Gamma \subseteq \partial \omega$, which is referred to as \emph{boundary strip domain}. 
Additionally, let $\TT_H$ be a mesh on $\omega$ with $\TT_H^\Gamma \coloneqq \TT_H|_\Gamma$. 
We recall the residual-based error estimator from~\cite{Carstensen1997,Carstensen2001} and the functional error estimator from~\cite{Kurz2021,Freiszlinger2025}.

\subsection{Residual error estimator}

We define the residual error estimator as
\begin{equation} \label{eq:resest}
 \rho_H(F;\psi) \coloneqq h_F^{1/2} \norm{\nabla_\Gamma(g - V\psi)}_F
 \quad \text{for all } F \in \TT_H^\Gamma \text{ and } \psi \in L^2(\Gamma),
\end{equation}
as well as for $\SS_H \subseteq \TT_H^\Gamma$
\begin{equation*}
 \rho_H(\SS_H;\psi) \coloneqq \Bigl(\sum_{F \in \SS_H} \rho_H(F;\psi)^2 \Bigr)^{1/2} \quad \text{ and } \quad \rho_H(\psi) \coloneqq \rho_H(\TT_H;\psi) \coloneqq \rho_H(\TT_H^\Gamma;\psi).
\end{equation*}
Note that for $\psi \in L^2(\Gamma)$, we have $V\psi \in H^1(\Gamma)$; see, e.g.,~\cite[Section~7]{McLean2000}.
For $\psi = \phi_H^\star$, we omit the second argument and simply write $\rho_H(F) \coloneqq \rho_H(F;\phi_H^\star)$, $\rho_H(\SS_H) \coloneqq \rho_H(\SS_H;\phi_H^\star)$, and $\rho_H \coloneqq \rho_H(\TT_H^\Gamma)$.
For $p = 0$, reliability of $\rho_H$, i.e., 
\begin{equation} \label{eq:resrel}
  \enorm{\phi^\star - \phi_H^\star} \leq \const{C}{rel}\rho_H,
\end{equation}
has been proven in~\cite{Carstensen1997} for $d = 2$ and~\cite{Carstensen2001} for $d = 3$, where $ \const{C}{rel} > 0$ depends only on $\Gamma$, $\kappa$-shape regularity, and, additionally for $d=3$, on the shapes of elements and patches in $\TT_H^\Gamma$.
For general $p \in \N_0$, reliability~\eqref{eq:resrel} is proved in~\cite{Feischl2013,Gantumur2013}, where $ \const{C}{rel}$ depends additionally on $p$.

\subsection{Functional error estimator} \label{sec:funcest}

Let $k \in \N$ and $q \in \N$ be fixed.  
For $z \in \NN_H^\Gamma$, we set $\omega_{H,z} \coloneqq \mathrm{int}\bigl(\omega_H^k[z]\bigr)$ and consider the usual finite element space $\SS^q(\TT_H[z]) \subset H^1(\omega_{H,z})$ of order $q \in \N$.
Given $\psi \in L^2(\Gamma)$, we employ the Scott--Zhang projection on the boundary $J_H \colon H^1(\Gamma) \to \SS^q(\TT_H^\Gamma)$ and define the discretized boundary residual $g_H(\psi) \coloneqq J_H(g - V\psi)$.
Recall the hat functions $(\zeta_{H,z})_{z \in \NN_H^\Gamma} \subset \SS^q(\TT_H^\Gamma)$ and define $\iota_{H,z} \coloneqq \omega_{H,z} \cap \NN_H^\Gamma$ as well as 
\begin{equation}
 \xi_{H,z} \coloneqq (\#\iota_{H,z})^{-1} \sum_{z' \in \iota_{H,z}} \zeta_{H,z'}  \in \SS^1(\TT_H^\Gamma).
\end{equation}
By the properties of the hat functions, we obtain
\begin{equation} \label{eq:unity}
 \sum_{z \in \NN_H^\Gamma} \xi_{H,z} = 1
\end{equation}
and $\xi_{H,z} = 0$ on $\partial(\omega_{H,z} \cap \Gamma)$, where the boundary is taken in $\Gamma$. For each $z \in \NN_H^\Gamma$, we consider the local problem of finding $w_{H,z}(\psi) \in \SS^{q + 1}(\TT_H[z])$ such that 
\begin{equation} \label{eq:auxprob}
 \begin{split}
  \dual{\nabla w_{H,z}(\psi)}{\nabla v_H}_{
  \omega_{H,z}} &= 0
  \quad \text{for all } v_H \in \SS^{q + 1}_0(\TT_H[z]), \\
  w_{H,z}(\psi)|_\Gamma &= 
  \begin{cases}
   \xi_{H,z} g_H(\psi) & \text{on } \partial  \omega_{H,z} \cap \Gamma, \\
   0 & \text{on } \partial \omega_{H,z} \setminus      \Gamma.
  \end{cases}
 \end{split}
\end{equation}
Because of the properties of the functions $(\xi_{H,z})_{z \in \NN_H^\Gamma}$, we are able to extend $w_{H,z}(\psi)$ by zero to the whole domain $\Omega$.
We define
\begin{equation} \label{eq:wdef}
 w_H(\psi) \coloneqq \sum_{z \in \NN_H^\Gamma}  w_{H,z}(\psi) \in H^1_{\varrho}(\Omega)
\end{equation}
and note that $w_H(\psi)|_\Gamma = g_H(\psi) = J_H(g - V\psi)$ due to~\eqref{eq:unity}. For $\psi = \phi_H^\star$ we will omit the argument and simply write $g_H \coloneqq g_H(\phi_H^\star)$, $w_{H,z} \coloneqq w_{H,z}(\phi_H^\star)$, and $w_H \coloneqq w_H(\phi_H^\star)$.

We define $\eta_H(T;\psi) \coloneqq \norm{\nabla w_H(\psi)}_T$, $\mathrm{osc}_H(T;\psi) \coloneqq h_T^{1/2}\norm{\nabla_\Gamma (1 - J_H)(g - V\psi)}_{\partial T \cap \Gamma}$, and the functional error estimator
\begin{equation} \label{eq:funcest}
 \mu_H(T;\psi) \coloneqq [\eta_H(T;\psi)^2 + \osc_H(T;\psi)^2]^{1/2}
 \quad \text{for } T \in \TT_H, \psi \in L^2(\Gamma).
\end{equation}
Given $\UU_H \subseteq \TT_H$, we abbreviate 
\begin{subequations} \label{eq:convention}
\begin{equation} 
 \mu_H(\UU_H;\psi) \coloneqq \Bigl(\sum_{T \in \UU_H} \mu_H(T;\psi)^2 \Bigr)^{1/2}.
\end{equation}
In the case of $\UU_H = \TT_H$ or $\psi = \phi_H^\star$, we simply write
\begin{equation} 
 \mu_H(\psi) \coloneqq \mu_H(\TT_H;\psi),
 \quad \mu_H(\UU_H) \coloneqq \mu_H(\UU_H;\phi_H^\star),
 \quad\mu_H \coloneqq \mu_H(\TT_H;\phi_H^\star).
\end{equation}
\end{subequations}
We use the analogous conventions~\eqref{eq:convention} also for the contributions $\eta_H$ and $\osc_H$ of $\mu_H$.

Setting $u_H \coloneqq \widetilde{V}\psi$ for $\psi \in L^2(\Gamma)$, Theorem~5 and Lemma~7 in~\cite{Kurz2021} yield reliability with respect to the $H^1$-seminorm in $\Omega$, i.e., 
\begin{equation} \label{eq:funcrelpot}
\norm{\nabla(u^\star - u_H)}_\Omega \leq \min_{\substack{w \in H^1(\Omega) \\ w|_\Gamma = g_H(\psi)}} \norm{\nabla w}_\Omega + \const{C}{osc}\mathrm{osc}_H(\psi)
 \lesssim \eta_H(\psi) + \mathrm{osc}_H(\psi) 
 \leq \sqrt{2} \, \mu_H(\psi),
\end{equation}
where $\const{C}{osc}>0$ and hence also the hidden constant depend only on $\kappa$-shape regularity of $\TT_H^\Gamma$ and $q$.
\begin{remark}
The functional error estimator $\mu_H$ is also reliable with respect to the energy error $\enorm{\phi^\star - \phi_H^\star}$ as one can see as follows$\colon$ 
Since $(u^\star - u_H^\star)|_\Gamma = V(\phi^\star - \phi_H^\star)$,~\eqref{eq:discprob} and~\eqref{eq:BIE} yield $\dual{u^\star - u_H^\star}{1}_\Gamma = 0$.
Hence, continuity of $V^{-1}$ and a Poincaré-type inequality establish
\begin{equation*}
 \enorm{\phi^\star - \phi_H^\star} 
 \lesssim \norm{(u^\star - u_H^\star)|_\Gamma}_{H^{1/2}(\Gamma)}
 \lesssim \norm{u^\star - u_H^\star}_{H^1_{\varrho}(\Omega)}
 \lesssim \norm{\nabla(u^\star - u_H^\star)}_\Omega,
\end{equation*}
where the hidden constants depend only on $\Omega$.
Together with~\eqref{eq:funcrelpot} for $\psi = \phi_H^\star$, this establishes reliability of $\mu_H$ with respect to the energy norm, i.e.,
\begin{equation} \label{eq:funcrel}
 \enorm{\phi^\star - \phi_H^\star} 
 \leq \const{C'}{rel} \mu_H,
\end{equation}
where $\const{C'}{rel} > 0$ depends only on $\Omega$, $\kappa$-shape regularity of $\TT_H^\Gamma$, and the polynomial degree $q$ of the auxiliary problems~\eqref{eq:auxprob}. \qed
\end{remark}
\subsection{Adaptive algorithm} \label{sec:adap}

The following adaptive algorithm is analyzed in the remainder of this work.
\begin{algorithm}[Adaptive algorithm]\label{algo:adap}
 \textbf{Input:}
 Conforming initial triangulation $\TT_0$ of the boundary strip domain $\omega \subseteq \Omega$, polynomial degree $p \in \N_0$ for the Galerkin BEM~\eqref{eq:discprob}, polynomial degree $q \in \N$ for the functional estimator in~\eqref{eq:auxprob}--\eqref{eq:funcest}, patch size $k \in \N$, and marking parameters $0 < \theta \leq 1$ and $\const{C}{mark} \geq 1$.
 \\
 \textbf{Loop:} For all $\ell = 0,1,2,\dots$, repeat the following steps:
 \begin{enumerate}[label = \rm(\roman*)]
  \item Extract the boundary mesh $\TT_{\ell}^{\Gamma} = \TT_\ell|_\Gamma$ from $\TT_{\ell}$ and solve~\eqref{eq:discprob} to obtain the Galerkin BEM solution $\phi_{\ell}^\star \in \PP^p(\TT_{\ell}^{\Gamma})$.
  \item Compute the error indicators $\mu_\ell(T)$ from~\eqref{eq:funcest} for all $T \in \TT_\ell$.
  \item Determine a set $\MM_\ell \subseteq \TT_{\ell}$ with up to a factor $\const{C}{mark}$ minimal cardinality such that
  \begin{align}\label{eq:doerfler}
   \theta \mu_\ell^2 \leq \mu_\ell(\MM_\ell)^2
  \end{align}
  \item Generate a new triangulation $\TT_{\ell+1}$ by use of NVB such that at least all elements in $\MM_{\ell}$ are refined, i.e., $\MM_\ell \subseteq \TT_\ell \setminus \TT_{\ell+1}$.
 \end{enumerate}
\end{algorithm}
\subsection{Main results} \label{sec:mainresults}
Results from~\cite{Feischl2013,Gantumur2013} show that under certain conditions on the adaptivity parameter $\theta$, an adaptive algorithm steered by the residual error estimator $\rho_H$ admits quasi-optimal convergence with respect to the number $\# \TT_H^\Gamma$ of boundary elements.
We aim to transfer this result to the functional error estimator $\mu_H$.
As in earlier works (\cite{Binev2004,Cascon2008,Feischl2013,Gantumur2013,Carstensen2014}), the key argument to prove this result is unconditional $R$-linear convergence of the functional error estimator $\mu_H$. 
The proof is postponed to Section~\ref{sec:linconv} below.
\begin{theorem} \label{thm:linconv}
 For any $0 < \theta \leq 1$, there exist constants $\const{C}{lin} > 0$ and $0 < \const{q}{lin} < 1$ depending only on the use of NVB, the polynomial degrees $p$, $q$, the patch size $k$, the dimension $d$, and the marking parameter $\theta$, such that Algorithm~\ref{algo:adap} guarantees $R$-linear convergence
 \begin{equation} \label{eq:linconv}
  \mu_{\ell + n} \leq \const{C}{lin} \const{q}{lin}^n\mu_\ell
  \quad \text{for all } \ell, n \in \N_0.
 \end{equation}
\end{theorem}
For the remainder of this work, we assume that $\TT_0$ is an initial mesh of the boundary strip domain $\omega$ and $\T \coloneqq \T(\TT_0)$.
For $N \in \N_0$, we define 
\begin{equation} \label{eq:Nmeshes}
 \begin{split}
  \T_N 
  &\coloneqq
  \set{\TT_H \in \T \given \# \TT_H - \# \TT_0 \leq N}, \\ 
  \mathbb{T}_N^\Gamma
  &\coloneqq
  \set{\TT_H \in \T \given \# \TT^\Gamma_{H} - \# \TT^\Gamma_0 \leq N},
 \end{split}
\end{equation}
i.e., all meshes that have at most $N$ (boundary) elements more than the initial mesh.
We note that $\T_N \subseteq \T_N^\Gamma$ and that both sets are indeed finite.
We say that the integral density $\phi^\star \in H^{-1/2}(\Gamma)$ belongs to the approximation class $\A_s^\rho$ for some $s > 0$, if there holds
\begin{equation} \label{eq:approxclass}
 \norm{\phi^\star}_{\A_s^\rho} \coloneqq \sup_{N \in \N_0} \bigl((N + 1)^s \min_{\TT_H \in \T_N^\Gamma} \rho_H \bigr) < \infty,
\end{equation}
i.e., the error estimator $\rho_H$ decays at least with rate $s > 0$ along a sequence of boundary meshes $\TT_H^\Gamma$ induced by optimal volume meshes $\TT_H$.
Analogously, we say that the corresponding potential $u^\star = \widetilde{V} \phi^\star$ belongs to the approximation class $\A_s^\mu$ for some $s > 0$, if there holds
\begin{equation} \label{eq:approxclassmu}
 \norm{u^\star}_{\A_s^\mu} \coloneqq \sup_{N \in \N_0} \bigl((N + 1)^s \min_{\TT_H \in \T_N} \mu_H \bigr) < \infty,
\end{equation}
i.e., the error estimator $\mu_H$ decays at least with rate $s > 0$ along a sequence of optimal volume meshes $\TT_H$.
The main theorem of this work states, first, that $\mu_H$ possesses the same approximation properties (with respect to $\# \TT_H$) as $\rho_H$ (with respect to $\#\TT_H^\Gamma$) and, second, if $\phi^\star$ can be approximated with the rate $s$ by $\rho_H$, then the adaptive algorithm steered by the functional error estimator $\mu_H$ will also achieve this rate provided the marking parameter $\theta$ is chosen sufficiently small.
The proof of the following theorem is postponed to Section~\ref{sec:optconv} below.
\begin{theorem} \label{thm:opt}
There exist constants $\const{c}{rate}, \const{C}{rate} > 0$ depending only on the use of NVB, the patch size $k$, the rate $s$, and the polynomial degree $q$, such that
\begin{equation} \label{eq:approxclasseq}
\const{c}{rate} \norm{\phi^\star}_{\A_s^\rho} \leq \norm{u^\star}_{\A_s^\mu} \leq \const{C}{rate} \norm{\phi^\star}_{\A_s^\rho}.
\end{equation}
Moreover, provided that $\norm{\phi^\star}_{\A_s^\rho} < \infty$, there exist constants $0 < \const{\theta}{opt} < 1$, $\const{C}{opt} > 0$, and $\const{C}{est} > 0$ such that Algorithm~\ref{algo:adap} with adaptivity parameter $0 < \theta \leq \const{\theta}{opt}$ guarantees
\begin{equation} \label{eq:opt}
  2^{-1/2} \min\set{1,\const{C}{osc}^{-1}} \norm{\nabla(u^\star - u_\ell^\star)}_\Omega
  \leq \mu_\ell \leq \const{C}{est} \rho_\ell \leq  \const{C}{opt} N_\ell^{-s}
  \leq \const{C}{opt} (N_\ell^\Gamma)^{-s},
\end{equation}
where $N_\ell \coloneqq \# \TT_\ell - \#\TT_0 + 1$ and $N_\ell^\Gamma \coloneqq \# \TT_\ell^\Gamma - \#\TT_0^\Gamma + 1$.
The constants $\const{\theta}{opt}$ and $\const{C}{opt}$ depend only on $\norm{\phi^\star}_{\A_s^\rho}$, the marking parameter $\theta$, the constants $\const{C}{lin}$ and $\const{q}{lin}$ from~\eqref{eq:linconv}, and the use of NVB, whereas $\const{C}{est}$ depends only on $\kappa$-shape regularity of the meshes $(\TT_\ell)_{\ell \in \N_0}$ and the patch size $k$.
\end{theorem}
\section{Local equivalences} \label{sec:equiv}
The residual error estimator $\rho_\ell$ satisfies $R$-linear convergence if it is used to steer an adaptive algorithm; see, e.g.,~\cite{Feischl2013,Gantumur2013,Carstensen2014}.
In order to transfer this result to the functional error estimator $\mu_H$, the following theorem will turn out to be the key argument. Throughout the remainder of this work, given $F \in \TT_H^\Gamma$, we denote by $T_F \in \TT_H$ the unique simplex satisfying $F \subset \partial T_F$.
\begin{theorem}[Local equivalence of \boldsymbol{$\rho_H$} and \boldsymbol{$\mu_H$}] \label{thm:equiv}
 The error estimators $\rho_H$ from~\eqref{eq:resest} and $\mu_H$ from~\eqref{eq:funcest} are locally equivalent in the sense that there exist constants $\const{C}{R2F}, \const{C}{F2R} > 0$ such that
 \begin{subequations} \label{eq:equiv}
 \begin{align}
   \rho_H(\SS_H;\psi) 
   &\leq \const{C}{R2F} \Bigl( \sum_{F \in \SS_H} \mu_H(T_F;\psi)^2 \Bigr)^{1/2}
   \quad \text{for all } \psi \in L^2(\Gamma) \text{ and } \SS_H \subseteq \TT_H^\Gamma, \label{eq:R2F} \\
   \mu_H(\UU_H)
   &\leq \const{C}{F2R} \Bigl( \sum_{\substack{F \in \TT_H^\Gamma \\ F \subset \omega_H^{2k + 1}[\UU_H]}} \rho_H(F)^2 \Bigr)^{1/2}
   \quad \text{for all } \UU_H \subseteq \TT_H. \label{eq:F2R}
 \end{align}
\end{subequations}
The constant $\const{C}{R2F}$ depends only on $\kappa$-shape regularity of $\TT_H$ and the polynomial degree $q$, whereas $\const{C}{F2R}$ depends only on $\kappa$-shape regularity of $\TT_H$, the polynomial degree $q$, and the patch size $k$.
\end{theorem}
\begin{proof}
 The proof is split into six steps.
 
 \textbf{Step~1 (\boldsymbol{$\rho_H(F;\psi) \lesssim \mu_H(T_F;\psi)$}).} Let $F \in \TT_H^\Gamma$ and $\psi \in L^2(\Gamma)$. The discrete trace inequality and $\kappa$-shape regularity yield that
\begin{equation*}
\begin{split}
 \rho_H(F;\psi)
 &= h_F^{1/2} \norm{\nabla_\Gamma(g - V\psi)}_F \\
 &\leq h_F^{1/2} \norm{\nabla_\Gamma J_H(g - V\psi)}_F + h_F^{1/2}\norm{\nabla_\Gamma(1 - J_H)(g - V\psi)}_F \\
 &\eqreff*{eq:wdef}{\leq} h_F^{1/2}\norm{\nabla_\Gamma w_H(\psi)|_\Gamma}_F + h_F^{1/2}\norm{\nabla_\Gamma(1 - J_H)(g - V\psi)}_{\partial T_F \cap \Gamma} \\
 &\lesssim \norm{\nabla w_H(\psi)}_{T_F} + \mathrm{osc}_H(T_F;\psi) 
 \lesssim \mu_H(T_F;\psi).
\end{split}
\end{equation*}
The hidden constant depends only on $\kappa$-shape regularity and the polynomial degree $q$. 
Given $\SS_H \subseteq \TT_H^\Gamma$, we can square and sum this estimate over all $F \in \SS_H$ to conclude the proof of~\eqref{eq:R2F}.

\textbf{Step~2 (\boldsymbol{$\mathrm{osc}_H(T) \lesssim \rho_H(\TT_H^\Gamma[\partial T \cap \Gamma])$}).}
Let $T \in \TT_H$. By the local stability properties of the Scott--Zhang projection and $\kappa$-shape regularity, we obtain that
\begin{equation*}
 \begin{split}
  \mathrm{osc}_H(T)
  &= h_T^{1/2}\norm{\nabla_\Gamma (1 - J_H)(g - V\phi_H^\star)}_{\partial T \cap \Gamma} \\
  &\lesssim h_T^{1/2} \norm{\nabla_\Gamma (g - V\phi_H^\star)}_{\Gamma_H[\partial T \cap \Gamma]}
  \lesssim \rho_H(\TT_H^\Gamma[\partial T \cap \Gamma]).
 \end{split}
\end{equation*}
The hidden constant depends only on $\kappa$-shape regularity and the polynomial degree $q$.

\textbf{Step~3 (Estimate of \boldsymbol{$\Vert \nabla w_{H,z} \rVert_{\omega_{H,z}}$}).}
Let $\overline{g}_{H,z} \in \SS^{q + 1}(\TT_H^{k}[z])$ be the natural lifting of $w_{H,z}|_\Gamma$, i.e., all interior degrees of freedom are set to zero and $\overline{g}_{H,z}|_\Gamma = w_{H,z}|_\Gamma$. We consider the problem of finding $\overline{w}_{H,z} \in \SS_0^{q + 1}(\TT_H^{k}[z])$ such that 
\begin{equation} \label{eq:shiftauxprob}
 \dual{\nabla \overline{w}_{H,z}}{\nabla v_H}_{\omega_{H,z}} = - \dual{\nabla \overline{g}_{H,z}}{\nabla v_H}_{\omega_{H,z}}
 \quad \text{for all } v_H \in \SS_0^{q + 1}(\TT_H^{k}[z]).
\end{equation}
It is well-known that~\eqref{eq:shiftauxprob} admits a unique solution and that there holds $w_{H,z} = \overline{w}_{H,z} + \overline{g}_{H,z}$. Moreover, we obtain that
\begin{equation*}
 \norm{\nabla w_{H,z}}_{\omega_{H,z}} 
 \leq \norm{\nabla \overline{w}_{H,z}}_{\omega_{H,z}} + \norm{\nabla \overline{g}_{H,z}}_{\omega_{H,z}}
 \leq 2 \norm{\nabla \overline{g}_{H,z}}_{\omega_{H,z}}.
\end{equation*}
\cite[Lemma~19]{Freiszlinger2025} allows us to estimate the natural lifting via
\begin{equation} \label{eq:wlzesti}
 \norm{\nabla w_{H,z}}_{\omega_{H,z}}^2
 \lesssim \norm{\nabla \overline{g}_{H,z}}_{\omega_{H,z}}^2 
 \lesssim  \sum_{\substack{F \in \TT_H^\Gamma \\ F \subset \omega_H^k[z]}} \bigl(h_F^{-1}\norm{\xi_{H,z} g_H}^2_F + \abs{\xi_{H,z} g_H}_{H^{1/2}(F)}^2\bigr).
\end{equation}
The hidden constant depends only on $\kappa$-shape regularity and the polynomial degree $q$.

\textbf{Step~4 (\boldsymbol{$h_F^{-1/2} \lVert \xi_{H,z} g_H \rVert_F \lesssim \rho_H(\TT_H^\Gamma[F])$}).}
We start by estimating the first term on the right-hand side of~\eqref{eq:wlzesti}. There holds
\begin{equation} \label{eq:hatesti}
 \norm{\xi_{H,z}}_{L^{\infty}(\Gamma)} \leq (\#\iota_{H,z})^{-1} \sum_{z' \in \iota_{H,z}} \norm{\zeta_{H,z'}}_{L^{\infty}(\Gamma)} = 1.
\end{equation}
Recall the $L^2(\Gamma)$-orthogonal projection $Q_H$ from~\eqref{eq:L2proj}. Together with the stability properties of the Scott-Zhang projection, the fact that $Q_H(g - V\phi_H^\star) = 0$ by~\eqref{eq:discprob}, and the properties of $Q_H$, estimate~\eqref{eq:hatesti} yields that
\begin{equation} \label{eq:L2esti}
\begin{split}
&h_F^{-1/2} \norm{\xi_{H,z} g_H}_F 
\leq h_F^{-1/2} \norm{J_H(g - V\phi_H^\star)}_{F}
\lesssim h_F^{-1/2} \norm{g - V\phi_H^\star}_{\Gamma_H[F]} \\
&\quad \eqreff{eq:discprob}{=} h_F^{-1/2} \norm{(1 - Q_H)(g - V\phi_H^\star)}_{\Gamma_H[F]}
\eqreff{eq:L2proj}{\lesssim} h_F^{1/2} \norm{\nabla_\Gamma (g - V\phi_H^\star)}_{\Gamma_H[F]} 
\lesssim \rho_H(\TT_H^\Gamma[F]).
\end{split}
\end{equation} 
In particular, there holds 
\begin{equation*} 
  \sum_{\substack{F \in \TT_H^\Gamma \\F \subset \omega_H^k[z]}} h_F^{-1} \norm{ \xi_{H,z} g_H}_F^2 
  \lesssim \sum_{\substack{F \in \TT_H^\Gamma \\F \subset \omega_H^{k + 1}[z]}} \rho_H(F)^2
\end{equation*}
 where the hidden constant depends only on $\kappa$-shape regularity.

\textbf{Step~5 (\boldsymbol{$|\xi_{H,z} g_H|_{H^{1/2}(F)} \lesssim \rho_H(\TT_H^\Gamma[F])$}).} Let $F \in \TT_H^\Gamma$ with $F \subseteq \partial \omega_{H,z} \cap \Gamma$ and $x,y \in F$. 
By the fundamental theorem of calculus and~\eqref{eq:hatesti}, we have 
\begin{equation*}
\begin{split}
 \abs{\xi_{H,z}(x)g_H(x) - \xi_{H,z}(y)g_H(y)} &\leq \abs{\xi_{H,z}(x)[g_H(x) - g_H(y)]} + \abs{[\xi_{H,z}(x) - \xi_{H,z}(y)]g_H(y)} \\
 &\leq \abs{g_H(x) - g_H(y)} + \norm{\nabla \xi_{H,z}}_{L^\infty(F)} \abs{x-y} \abs{g_H(y)}.
\end{split}
\end{equation*}
With 
\begin{equation*}
 \norm{\nabla \xi_{H,z}}_{L^{\infty}(F)}
 \leq (\#\iota_{H,z})^{-1} \sum_{z' \in \iota_{H,z}} \norm{\nabla \zeta_{H,z'}}_{L^\infty(F)} 
 \lesssim h_F^{-1}
\end{equation*}
we obtain that
\begin{equation} \label{eq:H12esti}
\begin{split}
 &\abs{\xi_{H,z} g_H}_{H^{1/2}(F)}^2 
 = \int_F \int_F \frac{\abs{\xi_{H,z}(x)g_H(x) - \xi_{H,z}(y)g_H(y)}^2}{\abs{x - y}^d} \d x \d y \\
 &\quad \lesssim \int_F \int_F \frac{\abs{g_H(x) - g_H(y)}^2}{\abs{x - y}^d} \d x \d y 
 + h_F^{-2} \int_F \abs{g_H(y)}^2 \int_F \abs{x - y}^{2 - d} \d x \d y \\
 &\quad = \abs{g_H}_{H^{1/2}(F)}^2 + h_F^{-2} \int_F \abs{g_H(y)}^2 \int_F \abs{x - y}^{2 - d} \d x \d y.
\end{split}
\end{equation}
It remains to estimate the term $\int_F \abs{x - y}^{2-d} \d x$. A transformation to the origin and an additional transformation $\Phi$ with respect to $(d-1)$-dimensional sphere coordinates satisfying $\mathrm{det}\bigl( D \Phi(x)\bigr) \lesssim \abs{x}^{d - 2}$ together with $\kappa$-shape regularity~\eqref{eq:shapereg} yield that
\begin{equation*}
\begin{split}
 \int_F \abs{x - y}^{2-d} \d x 
 \leq \int_{B^{d-1}_{\mathrm{diam}(F)}(0)} \abs{x}^{2-d} \d x 
 \lesssim 2\pi^{d-2} \mathrm{diam}(F)
 \lesssim h_F.
\end{split}
\end{equation*}
This, together with~\eqref{eq:H12esti} establishes that
\begin{equation*}
 \abs{\xi_{H,z} g_H}_{H^{1/2}(F)}^2 
 \lesssim \abs{g_H}_{H^{1/2}(F)}^2 + h_F^{-1} \norm{g_H}_{F}^2.
\end{equation*}
By~\eqref{eq:L2esti} in Step~4, the second term on the right-hand side can be estimated by $\rho_H(\TT_H^\Gamma[F])^2$. It remains to show that $\abs{g_H}_{H^{1/2}(F)} \lesssim \rho_H(\TT_H^\Gamma[F])$. With the properties of the Scott--Zhang operator, we obtain that
\begin{equation*}
 \begin{split}
 \abs{g_H}_{H^{1/2}(F)} 
 &= \abs{J_H(g - V\phi_H^\star)}_{H^{1/2}(F)}
 \lesssim h_F^{1/2} \norm{ \nabla_\Gamma (g - V\phi_H^\star)}_{\Gamma_H[F]} 
 \lesssim \rho_H(\TT_H^\Gamma[F]).
 \end{split}
\end{equation*}
All hidden constants depend only on $\kappa$-shape regularity.

\textbf{Step 6 (\boldsymbol{$\eta_H(T) \lesssim \rho_H(\TT_H^{2k + 1}[T]|_\Gamma)$}).}
Steps~4--5 together with~\eqref{eq:wlzesti} lead to the estimate
\begin{equation*}
 \norm{\nabla w_{H,z}}_{\omega_{H,z}}^2
 \lesssim \sum_{\substack{ F \in \TT_H^\Gamma \\ F \subset \omega_H^{k + 1}[z]}} \rho_H(F)^2,
\end{equation*}
which establishes that
\begin{equation*}
 \begin{split}
  \eta_H(T)^2 
  &\lesssim \sum_{z \in \NN_H \cap \omega_H^{k}[T]} \norm{\nabla w_{H,z}}_{\omega_{H,z}}^2
  \lesssim \sum_{z \in \NN_H \cap \omega_H^{k}[T]} \sum_{\substack{ F \in \TT_H^\Gamma \\ F \subset \omega_H^{k + 1}[z]}} \rho_H(F)^2
  \lesssim \sum_{\substack{ F \in \TT_H^\Gamma \\ F \subset \omega_H^{2k + 1}[T]}}\rho_H(F)^2,
 \end{split}
\end{equation*}
where the hidden constant depends only on $\kappa$-shape regularity, the polynomial degree $q$, and the patch size $k$. 
Given $\UU_H \subseteq \TT_H$, we may sum over all $T \in \UU_H$, to conclude the proof of~\eqref{eq:F2R} yielding $k$-dependency due to patch overlap.
This concludes the proof.
\end{proof}

\section{\texorpdfstring{Proof of Theorem~\ref{thm:linconv} ($\boldsymbol{R}$-linear convergence)}{Proof of R-linear convergence}} \label{sec:linconv}

\subsection{Equivalent mesh-size function}

In order to prove $R$-linear convergence of the functional error estimator $\mu_\ell$ under adaptive mesh-refinement, we recall the following result from~\cite[Proposition~8.6]{Carstensen2014}, which establishes the existence of a mesh-size function $\overline{H}$ such that bisection of elements in $T \in \TT_H$ leads to a uniform reduction of $\overline{H}$ on the patches $\omega_H^m[T]$ of $T$.
\begin{lemma}[Equivalent mesh-size function] \label{thm:equivmeshsize}
 For all $m \in \N$ and all $\TT_H \in \T$, there exists a function $\overline{H} \colon \TT_H \to \R_{>0}$ such that for all refinements $\TT_h \in \T(\TT_H)$ with mesh-size function $\overline{h}$, there hold the following properties:
 \begin{enumerate}[label=\rm(\roman*)]
  \item $\overline{H}(T) \leq h_T \leq \const{C}{eq} \overline{H}(T)$ for all $T \in \TT_H$;
  \item $\overline{h}(T') \leq \overline{H}(T)$ for all $T \in \TT_H$ and all $T' \in \TT_h$ with $T' \subseteq T$;
  \item $\overline{h}(T') \leq \const{q}{eq} \overline{H}(T)$ for all $T \in \TT_H^{m}[\TT_H \setminus \TT_h]$ and all $T' \in \TT_h$ with $T' \subseteq T$;
  \item $\overline{H}(T) = \overline{h}(T)$ for all $T \in \TT_H \setminus \TT_H^{m}[\TT_H \setminus \TT_h]$. 
 \end{enumerate}
 The constants $\const{C}{eq} \geq 1$ and $0 < \const{q}{eq} < 1$ depend only on $\kappa$-shape regularity of $\TT_H$, $m$, and the use of NVB.
\end{lemma}
With Lemma~\ref{thm:equivmeshsize}, we are able to define the \emph{modified} error estimator
\begin{equation} \label{eq:modest}
 \overline{\rho}_H(T;\psi) \coloneqq \overline{H}(T)^{1/2} \norm{\nabla_\Gamma(g - V\psi)}_{\partial T \cap \Gamma}
 \quad \text{for all } T \in \TT_H \text{ and } \psi \in L^2(\Gamma).
\end{equation}
We define $\overline{\rho}_H(\UU_H;\psi)$, $\overline{\rho}_H(\psi)$, etc. in the same way as in~\eqref{eq:convention}.
\begin{corollary} \label{cor:modequiv}
 For every $m \in \N$, the error estimators $\overline{\rho}_H$ and $\rho_H$ are locally equivalent in the sense that there exist constants $\const{C}{M2R},\const{C}{R2M} > 0$ such that 
 \begin{equation} \label{eq:modreseq}
  \begin{split}
   \overline{\rho}_H(\UU_H;\psi)
   &\leq \const{C}{M2R} \Bigl(\sum_{\substack{F \in \TT_H^\Gamma \\ F \subseteq x \bigcup \UU_H}} \rho_H(F;\psi)^2 \Bigr)^{1/2}
   \quad \text{for all } \psi \in L^2(\Gamma) \text{ and all } \UU_H \subseteq \TT_H, \\
   \rho_H(\SS_H;\psi)
   &\leq \const{C}{R2M} \Bigl(\sum_{F \in \SS_H} \overline{\rho}_H(T_F;\psi)^2 \Bigr)^{1/2}
    \quad \, \text{  for all } \psi \in L^2(\Gamma) \text{ and all } \SS_H \subseteq \TT_H^\Gamma.
  \end{split} 
 \end{equation}
 The constant $\const{C}{M2R}$ depends only on $\kappa$-shape regularity of $\TT_H$, whereas the constant $\const{C}{R2M}$ depends only on $\kappa$-shape regularity of $\TT_H$, $m$, and the use of NVB.
 In particular, $\overline{\rho}_H$ is reliable, i.e.,
 \begin{equation} \label{eq:modrel}
  \enorm{\phi^\star - \phi_H^\star} \leq \const{\overline{C}}{rel} \overline{\rho}_H
 \end{equation}
 with $\const{C}{rel}$ from~\eqref{eq:resrel} and $ \const{\overline{C}}{rel} = \const{C}{rel} \const{C}{R2M}$. Additionally, the error estimators $\overline{\rho}_H$ and $\mu_H$ are locally equivalent in the sense that
  \begin{subequations} \label{eq:modeq}
    \begin{align}
   \overline{\rho}_H(\UU_H;\psi) &\leq \const{\overline{C}}{R2F} \mu_H(\UU_H;\psi) 
   \qquad \ \ \  \text{for all } \psi \in L^2(\Gamma) \text{ and all } \UU_H \subseteq \TT_H,\label{eq:modR2F} \\
   \mu_H(\UU_H)
   &\leq \const{\overline{C}}{F2R} \overline{\rho}_H\bigl(\TT_H^{2k + 1}[ \UU_H ]\bigr)
   \quad \text{for all } \UU_H \subseteq \TT_H, \label{eq:modF2R}
 \end{align}
\end{subequations}
where $\const{\overline{C}}{R2F} \coloneqq \const{C}{R2F}\const{C}{M2R}$ and $\const{\overline{C}}{F2R} \coloneqq \const{C}{F2R}\const{C}{R2M}$ with $\const{C}{R2F}$ and $\const{C}{F2R}$ from Theorem~\ref{thm:equiv}.
\end{corollary}
\begin{proof}
  The proof is split into two steps.

 \textbf{Step~1 (Proof of (\ref{eq:modreseq}) and (\ref{eq:modrel})).}
 Let $F \in \TT_\ell^\Gamma$ and $T_F \in \TT_\ell$ such that $F \subseteq \partial T_F$. By $\kappa$-shape regularity, there holds $h_F \simeq \mathrm{diam}(F) \simeq \mathrm{diam}(T_F) \simeq h_{T_F}$. 
 Since NVB preserves $\kappa$-shape regularity, the hidden constants depend only on $\TT_0$.
 Hence, by Lemma~\ref{thm:equivmeshsize}(i), we obtain
 \begin{equation*}
  \overline{\rho}_H(T) 
  = \overline{H}(T)^{1/2} \norm{\nabla_\Gamma(g - V\psi)}_{\partial T \cap \Gamma}
  \simeq h_T^{1/2} \norm{\nabla_\Gamma(g - V\psi)}_{\partial T \cap \Gamma}
  \simeq \rho_H(\partial T \cap \Gamma).
 \end{equation*}
 Given $\SS_H \subseteq \TT_H^\Gamma$, we may square and sum this estimate over all $F \in \SS_H$ to conclude~\eqref{eq:modreseq}. Reliability of $\overline{\rho}_H$~\eqref{eq:modrel} follows immediately from~\eqref{eq:modreseq} and~\eqref{eq:resrel}.
 
 \textbf{Step~2 (Proof of~(\ref{eq:modeq})).}
The estimate~\eqref{eq:modR2F} follows immediately from \eqref{eq:R2F} in Theorem~\ref{thm:equiv} and~\eqref{eq:modreseq}, whereas the estimate~\eqref{eq:modF2R} follows immediately from~\eqref{eq:F2R} in Theorem~\ref{thm:equiv},~\eqref{eq:modreseq}, and the fact that $F \subset \omega_H^{2k + 1}[T]$ holds if and only if $T_F \in \TT_H^{2k + 1}[T]$.
\end{proof}

The following lemma collects two important properties, namely \emph{stability} and \emph{reduction} of $\overline{\rho}_H$, which were proven for $\rho_H$ in~\cite[Proposition~3.2]{Feischl2013} (for $p = 0$) and~\cite[Corollary~3.2]{Aurada2015a} (for general $p \in \N_0$).
The corresponding result for $\overline{\rho}_H$ follows from the observations in~\cite{Feischl2013,Aurada2015a} together with Lemma~\ref{thm:equivmeshsize}. For the convenience of the reader, we include the proof.
\begin{lemma} \label{thm:equivresprop}
 Let $\TT_h \in \T(\TT_H)$. Then, there exist constants $\const{C}{stab} > 0$ and $0 < \const{q}{red} < 1$ depending only on $\kappa$-shape regularity of $\TT_H$, $\Gamma$, the patch size $m$, the polynomial degree $p$, the dimension $d$, and the use of NVB such that the following statements hold for all $\psi_H \in \PP^p(\TT_H^\Gamma)$, $\phi_h \in \PP^p(\TT_h^\Gamma)$ and all $\UU_H \subseteq \TT_H \setminus \TT_H^m[\TT_H \setminus \TT_h]\colon$
 \begin{enumerate}[label=\rm(\roman*)]
  \item \textbf{Stability:} \quad $\abs{\overline{\rho}_H(\UU_H;\psi_H) - \overline{\rho}_h(\UU_H;\phi_h)} \leq \const{C}{stab} \enorm{\psi_H - \phi_h}$
  \smallskip
  \item \textbf{Reduction:} \qquad \quad \,$\overline{\rho}_h(\TT_h^{m}[\TT_h \setminus \TT_H];\psi_H) \leq \const{q}{red} \, \overline{\rho}_H(\TT_H^{m}[\TT_H \setminus \TT_h];\psi_H)$.
 \end{enumerate}
\end{lemma}
\begin{proof}
 The proof is split into three steps.

 \textbf{Step~1 (Proof of stability~(i)).} Let $\psi_H \in \PP^p(\TT_H^\Gamma)$, $\phi_h \in \PP^p(\TT_h^\Gamma)$, and $\UU_H \subseteq \TT_H \setminus \TT_H^m[\TT_H \setminus \TT_h]$. Note that $\TT_H \setminus \TT_H^{m}[\TT_H \setminus \TT_h] = \TT_h \setminus \TT_h^{m}[\TT_h \setminus \TT_H] \subseteq \TT_h \cap \TT_H$. Therefore, the triangle inequality, Lemma~\ref{thm:equivmeshsize}(iv), Lemma~\ref{thm:equivmeshsize}(i),  and noting that $\#\set{F \in \TT_h^\Gamma \given F \subset \partial T \cap \Gamma} \leq d +1$ for all $T \in \TT_h$ yield that
 \begin{equation*}
  \begin{split}
   &\abs{\overline{\rho}_H(\UU_H;\psi_H) - \overline{\rho}_h(\UU_H;\phi_h)}  
   \leq \Bigl( \sum_{T \in \UU_H} \overline{H}(T) \norm{\nabla_\Gamma V(\psi_H - \phi_h)}_{\partial T \cap \Gamma}^2 \Bigr)^{1/2} \\
   &\quad \lesssim \Bigl( \sum_{T \in \UU_H} h_T \norm{\nabla_\Gamma V(\psi_H - \phi_h)}_{\partial T \cap \Gamma}^2 \Bigr)^{1/2} 
   \lesssim \Bigl( \sum_{F \in \TT_h^\Gamma} h_F  \norm{\nabla_\Gamma V(\psi_H - \phi_h)}_F^2 \Bigr)^{1/2}.
  \end{split}
 \end{equation*}
 The term on the right-hand side of the equation above can be estimated by $\const{\widetilde{C}}{stab} \enorm{\psi_H - \phi_h}$ for some constant $\const{\widetilde{C}}{stab} > 0$ depending only on $\TT_0$, $\Gamma$, and $p$; see~\cite[Corollary~3.2]{Aurada2015a}.
 This concludes the proof of~(i).

 \textbf{Step~2.} Before proving reduction~(ii), we show that $\omega_h^{m}[\TT_h \setminus \TT_H] \subseteq \omega_H^{m}[\TT_H \setminus \TT_h]$.
 Let $x \in \omega_h^{m}[\TT_h \setminus \TT_H]$.
 Hence, there exists $T \in \TT_h^{m}[\TT_h \setminus \TT_H]$ such that $x \in T$.
 Since $T \in \TT_h^{m}[\TT_h \setminus \TT_H]$, there exist elements $T_0, \ldots, T_{j_0} \in \TT_h$ such that $T_0 = T$, $T_{j_0} \in \TT_h \setminus \TT_H$, and $T_j \in \TT_h[T_{j - 1}]$ for all $j = 1, \ldots , {j_0}$.
 Since $\TT_h \in \T(\TT_H)$, there exist $T_0', \ldots, T_{j_0}' \in \TT_H$ such that $T_j' \supseteq T_j$ for all $j = 0, \ldots, {j_0}$, where, additionally, $T_{j_0}' \supsetneqq T_{j_0}$, leading to $T_{j_0}' \in \TT_H \setminus \TT_h$.
 Hence, $T_j' \cap T_{j-1}' \supseteq T_j \cap T_{j-1} \neq \emptyset$ concludes that $x \in T_0 \subseteq T_0' \in \TT_H^{m}[\TT_H\setminus \TT_h]$.

 \textbf{Step~3 (Proof of reduction~(ii)).}
 Let $T \in \TT_H^{m}[\TT_H \setminus \TT_h]$ and $\psi_H \in \PP^p(\TT_H^\Gamma)$. 
 Then, Lemma~\ref{thm:equivmeshsize}(iii) implies that
 \begin{equation*}
  \begin{split}
   &\sum_{\substack{T' \in \TT_h \\ T' \subseteq T}} \overline{h}(T') \norm{\nabla_\Gamma(g - V\psi_H)}_{\partial T' \cap \Gamma}^2 
   \leq \const{q}{eq} \sum_{\substack{T' \in \TT_h \\ T' \subseteq T}} \overline{H}(T) \norm{\nabla_\Gamma(g - V\psi_H)}_{\partial T' \cap \Gamma}^2 \\
    &\qquad = \const{q}{eq} \overline{H}(T) \norm{\nabla_\Gamma(g - V\psi_H)}_{\partial T \cap \Gamma}^2 
    = \const{q}{eq} \overline{\rho}_H(T;\psi_H)^2.
  \end{split}
 \end{equation*}
 Summing over all $T \in \TT_H^{m}[\TT_H \setminus \TT_h]$ and noting that $\omega_h^{m}[\TT_h \setminus \TT_H] \subseteq \omega_H^{m}[\TT_H \setminus \TT_h]$ establishes 
 \begin{equation*}
  \begin{split}
   &\overline{\rho}_h(\TT_h^{m}[\TT_h \setminus \TT_H];\psi_H)^2
   = \sum_{T' \in \TT_h^m[\TT_h \setminus \TT_H]} \overline{h}(T')\norm{\nabla_\Gamma(g - V\psi_H)}_{\partial T' \cap \Gamma}^2 \\
   &\quad \leq \sum_{T \in \TT_H^m[\TT_h \setminus \TT_H]} \sum_{\substack{T' \in \TT_h \\ T' \subseteq T}} \overline{h}(T') \norm{\nabla_\Gamma(g - V\psi_H)}_{\partial T \cap \Gamma}^2 
   \leq \const{q}{eq}\overline{\rho}_H(\TT_H^{m}[\TT_H \setminus \TT_h];\psi_H)^2.
  \end{split}
 \end{equation*}
 This concludes the proof of reduction~(ii) with $\const{q}{red} \coloneqq \const{q}{eq}^{1/2}$.
\end{proof}
%
 %
 %
 %
 %

%

\subsection{\texorpdfstring{Proof of Theorem~\ref{thm:linconv} ($\bm{R}$-linear convergence)}{Proof of R-linear convergence}}

Before we are able to prove Theorem~\ref{thm:linconv}, we need the following summability result, which goes back to~\cite{Carstensen2014} and has since been refined in, e.g.,~\cite{Feischl2020,Bringmann2025,Bringmann2026}.
\begin{lemma}[Lemma~9 in \cite{Bringmann2026}] \label{lem:summability}
 Let $(a_\ell)_{\ell \in \N_0}, (b_\ell)_{\ell \in \N_0}$ be sequences in $\R_{\geq 0}$ such that there exist constants $0 < q' < 1$ and $C > 0$ with
 \begin{equation} \label{eq:summability}
  a_{\ell + 1} \leq q' a_\ell + b_\ell
  \quad \text{and} \quad 
  \sum_{\ell' = \ell}^{\ell + m} b_{\ell'}^2 \leq C a_\ell^2
  \quad \text{for all } \ell, m \in \N_0.
 \end{equation}
 Then, there exist constants $\const{C}{lin} > 0$ and $0 < \const{q}{lin} < 1$ such that
 \begin{equation} \label{eq:Rlin}
  a_{\ell + n} \leq \const{C}{lin} \const{q}{lin}^n a_\ell
  \quad \text{for all } \ell, n \in \N_0.
  \qquad \qed
    \end{equation}

\end{lemma}
\begin{proof}[\bfseries{Proof of Theorem~\ref{thm:linconv}}]
Let $\overline{h}_\ell$ be the equivalent mesh-size function from Lemma~\ref{thm:equivmeshsize} with respect to the patch size $m = 2k + 1$. Let $\overline{\rho}_\ell$ be the associated error estimator. 
The remaining proof is split into three steps.

\textbf{Step~1 (\boldsymbol{$\overline{\rho}_\ell$} inherits D\"orfler criterion).} 
By Corollary~\ref{cor:modequiv}, the D\"orfler marking criterion~\eqref{eq:doerfler} for $\mu_\ell$ implies
 \begin{equation*}
  \begin{split}
   \overline{\rho}_\ell^2 
   &\eqreff{eq:modR2F}{\leq}  \const{\overline{C}}{R2F}^2\mu_\ell^2
   \eqreff{eq:doerfler}{\leq} \const{\overline{C}}{R2F}^2 \theta^{-1} \mu_\ell(\MM_\ell)^2
   \eqreff{eq:modF2R}{\leq} \const{\overline{C}}{R2F}^2 \const{\overline{C}}{F2R}^2 \theta^{-1} \overline{\rho}_\ell(\TT_\ell^{2k + 1}[\MM_\ell])^2,
  \end{split}
 \end{equation*}
 leading to
 \begin{equation} \label{eq:resdoerfler}
  \theta' \overline{\rho}_\ell^2 \leq \overline{\rho}_\ell(\TT_\ell^{2k + 1}[\MM_\ell])^2
  \quad \text{with } \theta' \coloneqq \const{\overline{C}}{R2F}^{-2} \const{\overline{C}}{F2R}^{-2} \theta.
 \end{equation}
 We note that, due to~\eqref{eq:modeq} with $\UU_H = \TT_H$, it holds that $\const{\overline{C}}{R2F} \const{\overline{C}}{F2R} \geq 1$, ensuring that $0 < \theta' \leq 1$.

\textbf{Step~2 (Estimator reduction).}
We show that $\overline{\rho}_{\ell + 1}^2 \leq q' \overline{\rho}_\ell^2 + C \enorm{\phi_{\ell + 1}^\star - \phi_\ell^\star}^2$ with $0 < q' < 1$ and $C > 0$.
To this end, note that Lemma~\ref{thm:equivresprop} for $m = 2k + 1$, $\MM_\ell \subseteq \TT_\ell \setminus \TT_{\ell + 1}$, and Step~1 yield that
\begin{equation} \label{eq:estred}
 \begin{split}
  \overline{\rho}_{\ell + 1}(\phi_\ell^\star)^2 
  &= \overline{\rho}_{\ell + 1}(\TT_{\ell + 1} \setminus \TT_{\ell + 1}^{2k + 1}[\TT_{\ell + 1} \setminus \TT_\ell];\phi_\ell^\star)^2 + \overline{\rho}_{\ell + 1}(\TT_{\ell + 1}^{2k + 1}[\TT_{\ell + 1} \setminus \TT_\ell];\phi_\ell^\star)^2 \\
  &\leq \overline{\rho}_\ell(\TT_\ell \setminus \TT_\ell^{2k + 1}[\TT_\ell \setminus \TT_{\ell + 1}])^2 + \const{q}{red}^2 \overline{\rho}_\ell(\TT_\ell^{2k + 1}[\TT_\ell \setminus \TT_{\ell + 1}])^2 \\
  &= \overline{\rho}_\ell^2 - (1 - \const{q}{red}^2) \overline{\rho}_\ell(\TT_\ell^{2k + 1}[\TT_\ell \setminus \TT_{\ell + 1}])^2 \\
  &\leq \overline{\rho}_\ell^2 - (1 - \const{q}{red}^2) \overline{\rho}_\ell(\TT_\ell^{2k + 1}[\MM_\ell])^2 
  \eqreff{eq:resdoerfler}{\leq} [1 - (1 - \const{q}{red}^2) \theta']\overline{\rho}_\ell^2.
 \end{split}
\end{equation}
Let $\delta > 0$ with $q \coloneqq (1 + \delta)[1 - (1 - \const{q}{red}^2) \theta'] < 1$.
Then, Lemma~\ref{thm:equivresprop}(i), Young's inequality, and~\eqref{eq:estred} yield that
\begin{equation} \label{eq:quasicont}
 \begin{split}
  \overline{\rho}_{\ell + 1}^2
  &= \overline{\rho}_{\ell + 1}(\phi_{\ell + 1}^\star)^2
  \leq [\overline{\rho}_{\ell + 1}(\phi_\ell^\star) + \const{C}{stab} \enorm{\phi_{\ell + 1}^\star - \phi_\ell^\star}]^2 \\
  &\leq (1 + \delta) \overline{\rho}_{\ell + 1}(\phi_\ell^\star)^2 + (1 + \delta^{-1})\const{C}{stab}^2 \enorm{\phi_{\ell + 1}^\star - \phi_\ell^\star}^2 \\
  &\eqreff*{eq:estred}{\leq} (1 + \delta)[1 - (1 - \const{q}{red}^2) \theta'] \overline{\rho}_\ell^2 + (1 + \delta^{-1})\const{C}{stab}^2 \enorm{\phi_{\ell + 1}^\star - \phi_\ell^\star}^2.
 \end{split}
\end{equation}

\textbf{Step~3 (Application of Lemma~\ref{lem:summability}).}
Together with~\eqref{eq:modrel}, Galerkin orthogonality~\eqref{eq:discprob} yields that
\begin{equation} \label{eq:ortho}
 \sum_{\ell' = \ell}^{\ell + n} \enorm{\phi_{\ell' + 1}^\star - \phi_{\ell'}^\star}^2
 = \sum_{\ell' = \ell}^{\ell + n} \bigl(\enorm{\phi^\star - \phi_{\ell'}^\star}^2 - \enorm{\phi^\star - \phi_{\ell' + 1}^\star}^2\bigr)
 \leq \enorm{\phi^\star - \phi_\ell^\star}^2
 \eqreff{eq:modrel}{\leq} \const{\overline{C}}{rel}^2 \overline{\rho}_\ell^2.
\end{equation} 
Hence, Lemma~\ref{lem:summability} with $a_\ell \coloneqq \overline{\rho}_\ell^2$ and $b_\ell \coloneqq (1 + \delta^{-1})\const{C}{stab}^2 \enorm{\phi_{\ell + 1}^\star - \phi_\ell^\star}^2$ together with~\eqref{eq:quasicont}--\eqref{eq:ortho} yields the existence of constants $\const{\overline{C}}{lin}$ and $0 < \const{q}{lin} < 1$ such that
\begin{equation*}
 \overline{\rho}_{\ell + n}
 \leq \const{\overline{C}}{lin} \const{q}{lin}^n \overline{\rho}_\ell.
\end{equation*}
With $\const{C}{lin} \coloneqq \const{\overline{C}}{F2R} \const{\overline{C}}{R2F} \const{\overline{C}}{lin}$, Corollary~\ref{cor:modequiv} leads to
\begin{equation*}
 \mu_{\ell + n}
 \leq \const{C}{lin} \const{q}{lin}^n\mu_\ell.
\end{equation*}
This concludes the proof.
\end{proof}
\section{Proof of Theorem~\ref{thm:opt} (Optimal convergence)} \label{sec:optconv}
\subsection{Axioms of adaptivity}

In the present section, our goal is to prove that the functional error estimator $\mu_H$ leads to the same optimal convergence behavior as the residual error estimator $\rho_H$.
To this end, we first note that $\rho_H$ satisfies the following properties~\cite{Feischl2013,Gantumur2013}, called \emph{axioms of adaptivity} in~\cite{Carstensen2014}. 

\begin{lemma}[Axioms of adaptivity] \label{lem:axioms}
 Let $\TT_H \in \T$, $\TT_h \in \T(\TT_H)$, $\SS_H \subseteq \TT_H^\Gamma \cap \TT_h^\Gamma$, $\psi_H \in \PP^p(\TT_H^\Gamma)$, $\psi_h \in \PP^p(\TT_h^\Gamma)$, and $(\TT_\ell)_{\ell \in \N_0}$ be a sequence such that $\TT_{\ell + 1} \in \T(\TT_\ell)$ for all $\ell \in \N_0$. Then, the residual error estimator $\rho_H$ satisfies the following properties:
 \begin{enumerate}[label=\rm(\textup{A\arabic*})]
  \item \textbf{Stability on non-refined elements:} There exists a constant $\const{C}{stab} > 0$ such that
  \begin{equation} \label{eq:stab}
   \abs{\rho_h(\SS_H;\psi_h) - \rho_H(\SS_H;\psi_H)} 
   \leq \const{C}{stab} \enorm{\psi_h - \psi_H}.
  \end{equation}
  \item \textbf{Reduction on refined elements:} There exists a constant $0 < \const{q}{red} < 1$ such that
  \begin{equation} \label{eq:red}
   \rho_h(\TT_h^\Gamma \setminus \TT_H^\Gamma;\psi_H) 
   \leq \const{q}{red} \, \rho_H(\TT_H^\Gamma \setminus \TT_h^\Gamma;\psi_H).
  \end{equation}
  \item \textbf{Discrete reliability:} There exists a constant $\const{C}{dlr} > 0$ such that
  \begin{equation} \label{eq:dlr}
   \enorm{\phi_h^\star - \phi_H^\star}^2
   \leq \const{C}{dlr} \, \rho_H(\TT_H^\Gamma[\TT_H^\Gamma \setminus \TT_h^\Gamma];\phi_H^\star)^2.
  \end{equation}
  \item \textbf{Orthogonality:} With $\const{C}{rel}$ from~\eqref{eq:resrel}, there holds
  \begin{equation} \label{eq:quasiortho}
    \sum_{\ell' = \ell}^{\ell + n} \enorm{\phi_{\ell' + 1}^\star - \phi_{\ell'}^\star}^2
    \leq \const{C}{rel}^2 \rho_\ell^2
    \quad \text{for all } \ell, n \in \N_0.
  \end{equation}
 \end{enumerate}
 The constant $\const{C}{stab} > 0$ depends only on the boundary $\Gamma$, the initial mesh $\TT_0^\Gamma$, the dimension $d$, and the polynomial degree $p$. The constant $\const{q}{red} = 2^{-1/(d-1)} > 0$ depends only on the dimension $d$. The constant $\const{C}{dlr} > 0$ depends only on the boundary $\Gamma$, the use of NVB, and the polynomial degree $p$.
\end{lemma}
With Lemma~\ref{lem:axioms} at hand, the error estimator $\rho_\ell$ operates in the framework of~\cite{Carstensen2014}, allowing us to exploit the following properties of $\rho_\ell$. 
\begin{lemma} \label{lem:monopt}
 Let $\TT_H \in \T$ and $\TT_h \in \T(\TT_H)$. Then, the residual error estimator $\rho_\ell$ satisfies the following properties:
  \begin{enumerate}[label=\rm(\textup{\roman*})]
   \item \textbf{Quasi-monotonicity~\cite[Lemma~3.4]{Carstensen2014}:} There exists a constant $\const{C}{mon} > 0$ such that
   \begin{equation} \label{eq:mon}
    \rho_h \leq \const{C}{mon} \, \rho_H.
   \end{equation}
   \item \textbf{Optimality of Dörfler marking~\cite[Proposition~4.12]{Carstensen2014}:} There exist constants $0 < \theta_0, q_0 < 1$ such that there holds the implication
   \begin{equation} \label{eq:modcontr}
    \rho_h \leq q_0 \, \rho_H
    \quad \implies \quad
    \theta_0 \, \rho_H^2 \leq \rho_H(\TT_H^\Gamma[\TT_H^\Gamma \setminus \TT_h^\Gamma])^2.
   \end{equation}
  \end{enumerate}
  The constant $\const{C}{mon} > 0$ depends only on the boundary $\Gamma$, the initial mesh $\TT_0^\Gamma$, the dimension $d$, and the polynomial degree $p$, whereas the constants $\theta_0$ and $q_0$ depend only on the boundary $\Gamma$, the polynomial degree $p$, and the use of NVB.
\end{lemma}
We recall the following result concerning optimal complexity of NVB from~\cite{Binev2004,Stevenson2008,Karkulik2012}.
\begin{lemma}[Closure estimate] \label{lem:nvbopt}
  Let $(\TT_\ell)_{\ell \in \N_0}$ be a sequence of successive refinements in $\T$ in the sense that, for all $\ell \in \N_0$, $\TT_{\ell + 1} = \mathtt{refine}(\TT_\ell, \MM_\ell)$ for some set of marked elements $\MM_\ell \subseteq \TT_\ell$.
  Then, there exists a constant $\const{C}{nvb} > 0$ depending only on $\TT_0$ such that
  \begin{equation} \label{eq:nvbopt}
   \# \TT_\ell - \# \TT_0 \leq \const{C}{nvb} \sum_{\ell' = 0}^{\ell - 1} \# \MM_{\ell'}.
   \qquad \qed
  \end{equation}
\end{lemma}
Furthermore, we state the following result, which is found in~\cite{Stevenson2007,Cascon2008}.
\begin{lemma}[Overlay estimate] \label{lem:overlay}
 Let $\TT_H,\TT_h \in \T$. Then, the overlay
 \begin{equation} \label{eq:overlay}
  \TT_H \oplus \TT_h \coloneqq
  \set{T_H \cap T_h \given T_H \in \TT_H, T_h \in \TT_h, \abs{T_H \cap T_h} > 0}
 \end{equation}
 satisfies
 \begin{equation} \label{eq:overlayesti}
  \TT_H \oplus \TT_h \in \T
  \qquad \text{and} \qquad  
  \# (\TT_H \oplus \TT_h) \leq \# \TT_H + \# \TT_h - \# \TT_0.
  \qquad \qed
 \end{equation}
\end{lemma}
\subsection{Proof of Theorem~\ref{thm:opt}}

 The proof is split into four steps.

 \textbf{Step~1 (\boldsymbol{$\lVert \phi^\star \rVert_{\A_s^\rho} \lesssim \lVert u^\star \rVert_{\A_s^\mu}$}).}
 Recall the notation from Section~\ref{sec:mainresults} and note that $\T_N \subseteq \T_N^\Gamma$. Therefore, Theorem~\ref{thm:equiv} shows that
 \begin{equation} \label{eq:aceq1}
  \norm{\phi^\star}_{\A_s^\rho}
  = \sup_{N \in \N_0} \bigl( (N + 1)^s \min_{\TT_H \in \T_N^\Gamma} \rho_H \bigr)
  \eqreff*{eq:R2F}{\leq} \const{C}{R2F} \sup_{N \in \N_0} \bigl( (N + 1)^s \min_{\TT_H \in \T_N} \mu_H \bigr)
  = \const{C}{R2F} \norm{u^\star}_{\A_s^\mu}.
 \end{equation}

 \textbf{Step~2.} We show that, for all $N \in \N_0$ and all $\TT_h \in \T_N^\Gamma$, there exists a mesh $\TT_H \in \T_{\lceil 2\const{C}{nvb} N \rceil}$ such that $\TT_H^\Gamma = \TT_h^\Gamma$, where $\const{C}{nvb} > 0$ is the constant from the closure estimate of Lemma~\ref{lem:nvbopt}.
 Let $N \in \N_0$ and $\TT_h \in \T_N^\Gamma$.
 We construct a sequence of successive refinements starting with $\TT_0$ and setting $\TT_{j + 1} \coloneqq \mathtt{refine}(\TT_j,\MM_j)$, where
 \begin{equation*}
  \MM_j \coloneqq \set{T_F \in \TT_j \given F \in \TT_j^\Gamma \setminus \TT_h^\Gamma \text{ such that there exists } F' \in \TT_h^\Gamma \text{ with } F' \subsetneqq F}.
 \end{equation*}
 Since $\TT_h$ is obtained by finitely many steps of NVB, there exists $j_0 \in \N_0$ such that $\MM_j = \emptyset$ and $\TT_j = \TT_{j + 1}$ for all $j \geq j_0$.
 We set $\TT_H \coloneqq \TT_{j_0}$ and note that, due to $\TT_{j + 1}$ being the coarsest conforming refinement of $\TT_j$ such that all elements in $\MM_j$ have been refined (see~\cite[Theorem~5.1]{Stevenson2008}), and since $F' \subsetneqq F$ implies $T_{F'} \subsetneqq T_F$ for all $F \in \TT_j$ and all $F' \in \TT_h^\Gamma$, $\TT_H$ is not finer than $\TT_h$, which also implies that $\TT_H^\Gamma$ is not finer than $\TT_h^\Gamma$.
 On the other hand, since $\MM_{j_0} = \emptyset$, it holds that either $\TT_{j_0}^\Gamma \setminus \TT_h^\Gamma = \emptyset$, which immediately yields $\TT_h^\Gamma = \TT_H^\Gamma$, or for all $F \in \TT_{j_0}^\Gamma \setminus \TT_h^\Gamma$ and all $F' \in \TT_H^\Gamma$ with $\abs{F \cap F'} > 0$, there holds $F \subseteq F'$. The latter implies that $\TT_H^\Gamma$ is not coarser than $\TT_h^\Gamma$.
 Hence, $\TT_H^\Gamma$ is neither coarser nor finer than $\TT_h^\Gamma$, leading to $\TT_H^\Gamma = \TT_h^\Gamma$.

 It remains to show that $\TT_H \in \T_{\lceil 2\const{C}{nvb} N \rceil}$.
 Due to the definition of $\MM_j$ and due to the fact that $T_F$ has to be bisected at most two times for $F$ to be bisected, there holds
 \begin{equation} \label{eq:markest}
  \# \MM_j + \# \MM_{j + 1} \leq 2 (\# \TT_{j + 2}^\Gamma - \# \TT_j^\Gamma)
  \quad \text{for all } j \in \N_0.
 \end{equation}
 Note that $\MM_j = \emptyset$ and $\TT_{j_0} = \TT_j$ for all $j \geq j_0$.
 Hence, with $\TT_h \in \T_N^\Gamma$, Lemma~\ref{lem:nvbopt} establishes that
 \begin{equation*}
  \begin{split}
  &\# \TT_H - \# \TT_0
  \eqreff*{eq:nvbopt}{\leq} \const{C}{nvb} \sum_{j = 0}^{j_0-1} \# \MM_j
  \leq \const{C}{nvb} \sum_{j = 0}^{\lceil (j_0-2)/2 \rceil} (\# \MM_{2j} + \# \MM_{2j + 1}) \\
  &\quad \eqreff*{eq:markest}{\leq} 2 \const{C}{nvb} \sum_{j = 0}^{\lceil (j_0 - 2)/2 \rceil} (\# \TT_{2(j + 1)}^\Gamma - \# \TT_{2j}^\Gamma)
  = 2 \const{C}{nvb} (\# \TT_{j_0}^\Gamma - \# \TT_0^\Gamma) \\
  &\quad = 2\const{C}{nvb} (\# \TT_h^\Gamma - \# \TT_0^\Gamma)
  \leq 2 \const{C}{nvb} N
  \leq \lceil 2 \const{C}{nvb} N \rceil.
  \end{split}
 \end{equation*}
 This shows that $\TT_H \in \T_{\lceil 2 \const{C}{nvb} N \rceil}$.

 \textbf{Step~3 (\boldsymbol{$\lVert u^\star \rVert_{\A_s^\mu} \lesssim \lVert u^\star \rVert_{\A_s^\rho}$}).}
 Let $M \in \N_0$ and define $N_M \coloneqq \max \set{N' \in \N_0 \given \lceil 2 \const{C}{nvb} N' \rceil \leq M}$. This leads to $\lceil 2 \const{C}{nvb} N_M \rceil \leq M < \lceil 2 \const{C}{nvb} (N_M + 1) \rceil$.
 For every $\TT_h \in \T_{N_M}^\Gamma$, Step~2 provides $\TT_H \in \T_{\lceil 2 \const{C}{nvb} N_M \rceil}$ with $\TT_H^\Gamma = \TT_h^\Gamma$.
 Since the residual error estimator depends only on the boundary mesh, there holds $\rho_H = \rho_h$. 
 Together with $\T_{\lceil 2 \const{C}{nvb} N_M \rceil} \subseteq \T_M$ and $N + 1 \leq 2N$ for all $N \in \N$, Theorem~\ref{thm:equiv} yields that
 \begin{equation*}
  \begin{split}
   &(M + 1)^s \min_{\TT_H \in \T_M} \mu_H
   \eqreff{eq:F2R}{\leq} \const{C}{F2R} (M + 1)^s  \min_{\TT_H \in \T_M} \rho_H
   \leq \const{C}{F2R} \bigl(M + 1\bigr)^s \min_{\TT_H \in \T_{\lceil 2 \const{C}{nvb} N_M \rceil}} \rho_H \\
   &\qquad \leq \const{C}{F2R}\lceil 2 \const{C}{nvb} (N_M + 1) \rceil^s \min_{\TT_h \in \T_{N_M}^\Gamma} \rho_h
   \leq  \const{C}{F2R}(2\const{C}{nvb} + 1)^s (N_M + 1)^s \min_{\TT_h \in \T_{N_M}^\Gamma} \rho_h.
  \end{split}
 \end{equation*}
 Taking the supremum over $M \in \N_0$, we establish that
 \begin{equation} \label{eq:aceqfinal}
  \begin{split}
  \norm{u^\star}_{\A_s^\mu}
  &\leq (2\const{C}{nvb} + 1)^s \const{C}{F2R} \sup_{M \in \N_0} (N_M + 1)^s \min_{\TT_h \in \T_{N_M}^\Gamma} \rho_h \leq (2\const{C}{nvb} + 1)^s \const{C}{F2R} \norm{\phi^\star}_{\A_s^\rho}.
  \end{split}
 \end{equation}
 The estimates~\eqref{eq:aceq1} and~\eqref{eq:aceqfinal} conclude the proof of~\eqref{eq:approxclasseq} with $\const{c}{rate} = \const{C}{R2F}^{-1}$ and $\const{C}{rate} \coloneqq (2\const{C}{nvb} + 1)^s \const{C}{F2R}$.

 \textbf{Step~4 (Proof of (\ref{eq:opt})).}
 First, suppose that $\ell = 0$. Since $N_0 = 1$, it holds that
 \begin{equation*}
  N_0^s \rho_0 = \rho_0 = \min_{\TT_H \in \T_0^\Gamma} \rho_H \leq \sup_{N \in \N_0} \bigl( (N + 1)^s \min_{\TT_H \in \T_N} \rho_H \bigr) = \norm{\phi^\star}_{\A_s^\rho} < +\infty.
 \end{equation*}
 Hence, Theorem~\ref{thm:equiv} yields that
 \begin{equation*}
  \mu_0
  \eqreff{eq:F2R}{\lesssim} \rho_0
  \lesssim N_0^{-s}
  = (N_0^\Gamma)^{-s},
 \end{equation*}
 which, together with reliability~\eqref{eq:funcrelpot}, concludes the proof for $\ell = 0$. 
 
 Let $\ell \geq 1$ for the remainder of the proof.
 Recall $0 < q_0 < 1$ and $\const{C}{mon} > 0$ from Lemma~\ref{lem:monopt}. Without loss of generality, we assume that $\rho_\ell > 0$ and $\norm{\phi^\star}_{\A_s^\rho} > 0$. Let $0 < \delta < q_0\const{C}{mon}^{-1}$ and set $\eps \coloneqq \delta \rho_\ell  > 0$ as well as $N \coloneqq \lceil \eps^{-1/s} \, \norm{\phi^\star}_{\A_s^\rho}^{1/s} \rceil \in \N$. Since $\norm{\phi^\star}_{\A_s^\rho} < \infty$, there exists a mesh $\TT_\eps \in \T$ such that 
 \begin{equation} \label{eq:mesheps}
  \# \TT_\eps^\Gamma - \# \TT_0^\Gamma 
  \leq N = \lceil \eps^{-1/s} \, \norm{\phi^\star}_{\A_s^\rho}^{1/s} \rceil 
  \lesssim \eps^{-1/s} 
 \end{equation} 
 and
 \begin{equation} \label{eq:reseps}
  \rho_\eps 
  \leq (N + 1)^{-s} \norm{\phi^\star}_{\A_s^\rho}
  \leq N^{-s} \norm{\phi^\star}_{\A_s^\rho}
  \leq (\eps^{-1/s} \, \norm{\phi^\star}_{\A_s^\rho}^{1/s})^{-s} \norm{\phi^\star}_{\A_s^\rho}
  = \eps.
 \end{equation}
 We define $\TT_h \coloneqq \TT_\eps \oplus \TT_\ell$ as the overlay of $\TT_\eps$ and $\TT_\ell$; see~\eqref{eq:overlay}.
 Due to Lemma~\ref{lem:monopt}(i) and the choice of $\eps$, we obtain that
 \begin{equation}
  \begin{split}
  \rho_h
  &\eqreff*{eq:mon}{\leq} \const{C}{mon} \rho_\eps
  \eqreff*{eq:reseps}{\leq} \const{C}{mon} \eps
  = \const{C}{mon} \delta \rho_\ell
  < q_0 \rho_\ell.
  \end{split}
 \end{equation}
 Hence, Lemma~\ref{lem:monopt}(ii) yields that
 \begin{equation} \label{eq:dorflerappl}
  \theta_0 \rho_\ell^2 \leq \rho_\ell(\TT_\ell^\Gamma[\TT_\ell^\Gamma \setminus \TT_h^\Gamma])^2.
 \end{equation}
Given $0 < \theta \leq \const{\theta}{opt} \coloneqq \const{C}{F2R}^{-2} \const{C}{R2F}^{-2} \theta_0$, Theorem~\ref{thm:equiv} establishes that
 \begin{equation*}
  \begin{split}
  \theta \mu_\ell^2
  &\eqreff{eq:F2R}{\leq} \const{C}{F2R}^2 \theta \rho_\ell^2 
  \leq \const{C}{F2R}^2 \const{\theta}{opt} \rho_\ell^2
  = \const{C}{R2F}^{-2} \theta_0 \rho_\ell^2
  \eqreff{eq:dorflerappl}{\leq} \const{C}{R2F}^{-2} \rho_\ell(\TT^\Gamma_\ell[\TT^\Gamma_\ell \setminus \TT^\Gamma_h])^2.
  \end{split}
 \end{equation*}
 For $F \in \TT_\ell^\Gamma[\TT_\ell^\Gamma \setminus \TT_h^\Gamma]$, there exists $F' \in \TT_\ell^\Gamma \setminus \TT_h^\Gamma$ such that $F \cap F' \neq \emptyset$.
 Hence, $T_F \cap T_{F'} \neq \emptyset$ and $T_{F'} \in \TT_\ell \setminus \TT_h$ for some $T_F, T_{F'} \in \TT_\ell$ with $F \subseteq \partial T_F$ and $F' \subseteq \partial T_{F'}$.
 Hence, the inequality above together with~\eqref{eq:R2F} yields that
 \begin{equation*}
  \theta \mu_\ell^2
  \leq \const{C}{R2F}^{-2} \rho_\ell(\TT_\ell^\Gamma[\TT_\ell^\Gamma \setminus \TT_h^\Gamma])^2
  \eqreff{eq:R2F}{\leq} \mu_\ell(\TT_\ell[\TT_\ell^\Gamma \setminus \TT_h^\Gamma])^2.
 \end{equation*}
 Since $\MM_\ell$ has quasi-minimal cardinality, the overlay estimate~\eqref{eq:overlayesti} of Lemma~\ref{lem:overlay} and Theorem~\ref{thm:equiv} prove that
 \begin{equation*}
  \begin{split}
  \#\MM_\ell
  &\lesssim \# \TT_\ell[\TT^\Gamma_\ell \setminus \TT^\Gamma_h]
  \lesssim \# (\TT^\Gamma_\ell \setminus \TT^\Gamma_h)
  \lesssim \# \TT^\Gamma_h - \# \TT^\Gamma_\ell \\
  &\eqreff*{eq:overlayesti}{\leq} \# \TT^\Gamma_\eps - \# \TT^\Gamma_0
  \eqreff{eq:mesheps}{\lesssim} \eps^{-1/s}
  \eqreff{eq:reseps}{\lesssim} \rho_\ell^{-1/s}
  \eqreff{eq:F2R}{\lesssim} \mu_\ell^{-1/s}.
  \end{split}
 \end{equation*}
 The closure estimate~\eqref{eq:nvbopt} of Lemma~\ref{lem:nvbopt} and Theorem~\ref{thm:linconv} thus yield that
 \begin{equation*}
  \begin{split}
   N_\ell
   \lesssim \# \TT_\ell - \# \TT_0
   \lesssim \sum_{\ell' = 0}^{\ell - 1} \# \MM_{\ell'} 
   \lesssim \sum_{\ell' = 0}^{\ell - 1} \mu_{\ell'}^{-1/s}
   \eqreff{eq:linconv}{\lesssim} \mu_\ell^{-1/s} \sum_{\ell' = 0}^{\ell - 1} \const{q}{lin}^{(\ell - \ell')/s}
  \lesssim \mu_\ell^{-1/s}.
  \end{split}
 \end{equation*}
 This, together with Theorem~\ref{thm:equiv} and reliability~\eqref{eq:funcrelpot}, establishes
 \begin{equation*}
  \norm{\nabla(u^\star - u_\ell^\star)}_\Omega 
  \lesssim \mu_\ell
  \lesssim \rho_\ell
  \lesssim N_\ell^{-s} 
  \leq (N_\ell^\Gamma)^{-s},
 \end{equation*}
 which concludes the proof.
\qed

\section{Numerical experiments} \label{sec:numerics}

This section presents some numerical experiments in $2$D that illustrate the performance and accuracy of the proposed functional error estimator and the corresponding adaptive algorithm.
All computations are carried out using the MATLAB toolbox HILBERT~\cite{Aurada2014} for BEM with $p \in \set{0,1}$.
Throughout, we consider Algorithm~\ref{algo:adap} for uniform ($\theta = 1$) and adaptive ($0 < \theta < 1$) mesh refinement.

\subsection{Example 1 (Square domain with hole, smooth solution)} \label{subsec:squareex}

We consider the Laplace--Dirichlet equation~\eqref{eq:Laplace} on the domain $\Omega \coloneqq (-1/4,1/4)^2 \setminus [-1/52,1/52]^2$ with prescribed solution 
\begin{equation*}
 u^\star(x,y) \coloneqq (x + y)/(x^2 + y^2),
\end{equation*}
which is smooth in $\Omega$ but exhibits a (non-integrable) singularity at the origin.
We note that $\Gamma$ fulfills the scaling condition $\diam(\Gamma) = 1/\sqrt{2} < 1$.
We start Algorithm~\ref{algo:adap} with an initial volume mesh consisting of $672$ triangles and $56$ boundary edges.
We note that due to the hole in $\Omega$ the initial mesh is chosen rather fine to resolve the geometry.
The initial mesh and some adaptively generated meshes are shown in Figure~\ref{fig:smeshes}.
\begin{figure}[!ht]
 \resizebox{\textwidth}{!}{
   \subfloat{
    \stackunder[10pt]{\includegraphics[scale = 0.5]{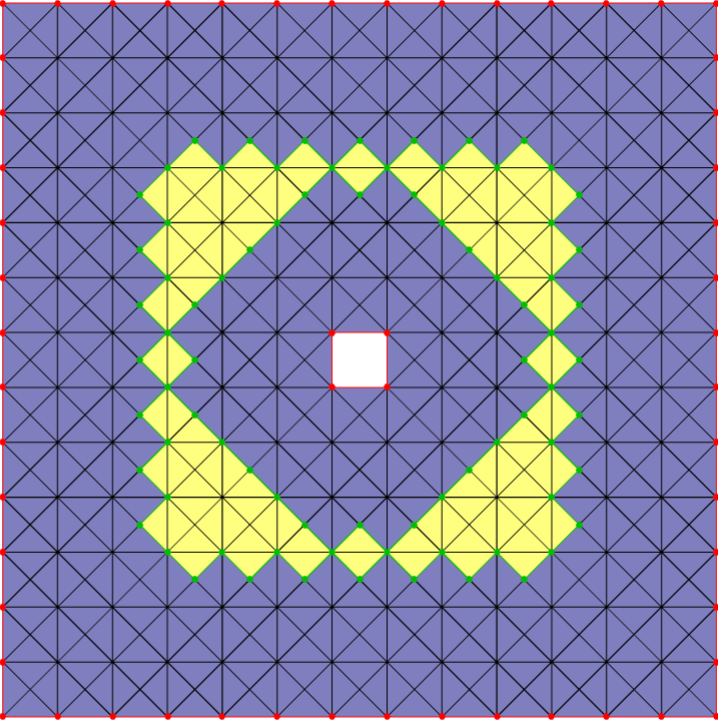}}{{\large $\# \TT_\ell^\Gamma = 56$, $\ell = 0$}}
   }
   \subfloat{
    \stackunder[10pt]{\includegraphics[scale = 0.5]{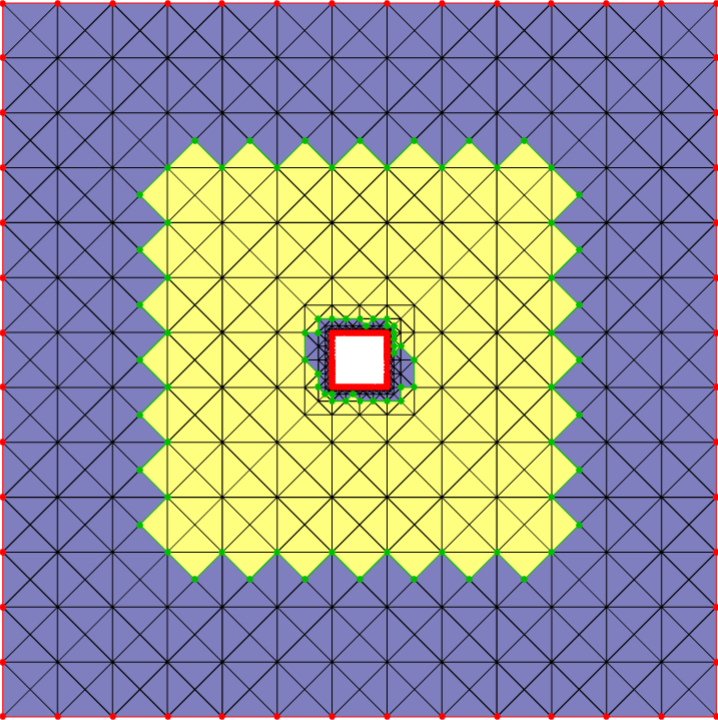}}{{\large $\# \TT_\ell^\Gamma = 186$, $\ell = 30$}}
   }
   
   \subfloat{
   \stackunder[10pt]{\includegraphics[scale = 0.5]{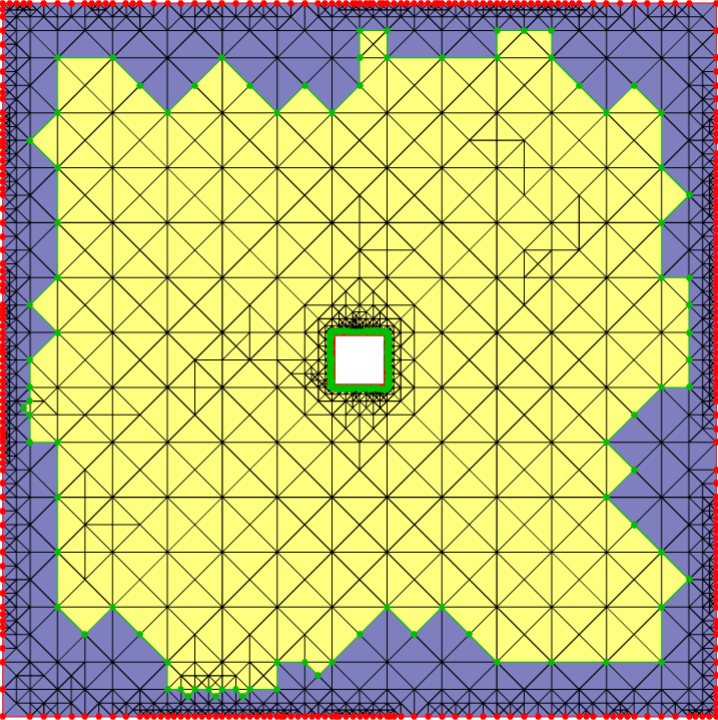}}{{\large $\# \TT_\ell^\Gamma = 1898$, $\ell = 100$}}
   }
  }
   \caption{Meshes generated by Algorithm~\ref{algo:adap} in Example~\ref{subsec:squareex} for $\theta = 0.4$ and $k = 3$. The triangles of the patches $\omega_{\ell,z}$ are depicted in \ul{blue}\setulcolor{yellow}, while the remaining triangles are indicated in \ul{yellow}\setulcolor{red}. The outer boundary $\Gamma$ is shown in \ul{red}\setulcolor{green} and the inner boundary of the union of the patches in \ul{green}\setulcolor{blue}.}
    \label{fig:smeshes}
\end{figure}
Figure~\ref{fig:srates} shows the total upper bound $\mu_\ell$ from~\eqref{eq:funcest} for $p = 0$~(left) and $p = 1$~(right) for patch size $k = 3$ and different adaptivity parameters $\theta \in \set{0.2,0.4,0.6,0.8,1}$. 
\begin{figure}[!ht]
  \resizebox{\textwidth}{!}{
   \subfloat{
    \includegraphics[scale = 0.5]{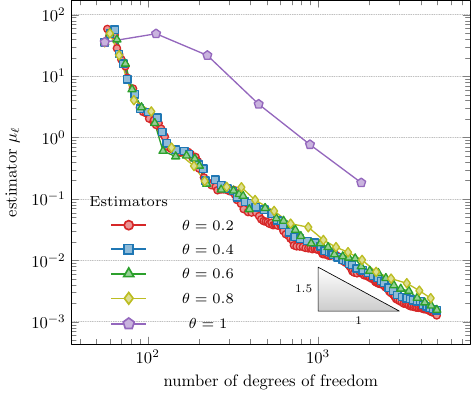}
   }
   \subfloat{
    \includegraphics[scale = 0.5]{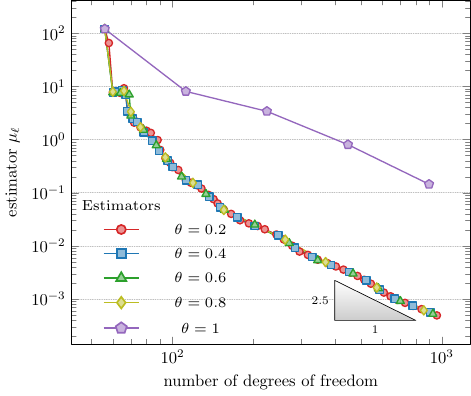}
   }
  }
   \caption{Convergence rates of the full error estimator $\mu_\ell$ from~\eqref{eq:funcest} generated by Algorithm~\ref{algo:adap} in Example~\ref{subsec:squareex} for $k = 3$, different marking parameters $0 < \theta \leq 1$, $p = 0$~(left) and $p = 1$~(right).}
   \label{fig:srates}
 \end{figure}
We observe that for all $\theta$, the error estimator $\mu_\ell$ decays with the optimal convergence rate $\OO((\# \TT_\ell^\Gamma)^{-(p + 3/2)})$ for $p \in \set{0,1}$, although uniform refinement ($\theta = 1$) leads to much larger estimator values than adaptive refinement ($0 < \theta < 1$).
A comparison between the error estimator $\eta_\ell$ from Section~\ref{sec:funcest}, the functional error estimator obtained by globally solving the auxiliary problem~\cite{Kurz2021,Freiszlinger2025}, the residual error estimator $\rho_\ell$ from~\eqref{eq:resest}, the Faermann error estimator~\cite{Faermann2000,Faermann2002}, the $h-h/2$ error estimator~\cite{Mund1998,Ferraz-Leite2008}, and the exact error $\norm{\nabla(u^\star - u_\ell^\star)}_\Omega$ for $p = 0$, $k = 3$, and $\theta = 0.4$ is shown in Figure~\ref{fig:scomp}~(left).
The corresponding experimental reliability constants $\mathtt{estimator} / \norm{\nabla(u^\star - u_\ell^\star)}_\Omega$ are shown in Figure~\ref{fig:scomp}~(right).
\begin{figure}[!ht]
  \resizebox{\textwidth}{!}{
   \subfloat{
    \includegraphics{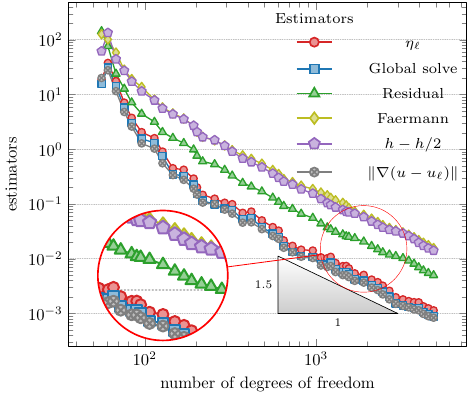}
   }
   \subfloat{
    \includegraphics{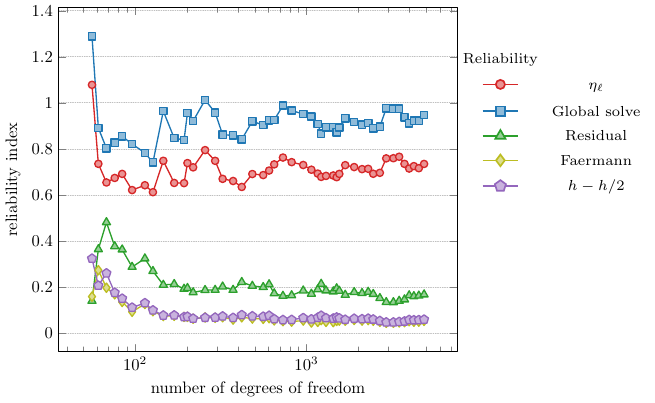}
   }
  }
   \caption{Comparison between $\eta_\ell$, the functional estimator obtained by a global solve, the residual error estimator $\rho_\ell$, the Faermann error estimator, the $h-h/2$ error estimator, and the exact error in terms of convergence rates~(left) and experimental reliability constants~(right) for $p = 0$, $k = 3$, and $\theta = 0.4$ in Example~\ref{subsec:squareex}.}
   \label{fig:scomp}
 \end{figure}
Here, the exact error is computed by $\norm{\nabla(u^\star - u_\ell^\star)}_\Omega \approx \norm{\nabla I_\ell(u^\star - u_\ell^\star)}_\Omega$, where $I_\ell$ denotes the $\SS^2(\TT_\ell)$-nodal interpolation operator.
We observe that the functional error estimators, in contrast to the remaining estimators, exhibit excellent quality in terms of approximating the exact error $\norm{\nabla(u^\star - u_\ell^\star)}_\Omega$ with only minimal differences between the local and global functional error estimators.
Figure~\ref{fig:sk} shows the exact error as well as the error estimator $\eta_\ell$ for $\theta = 0.4$, different patch sizes $k \in \set{1,2,3,4,5}$, $p = 0$~(left) and $p = 1$~(right).
\begin{figure}[!ht]
  \resizebox{\textwidth}{!}{
   \subfloat{
    \includegraphics[scale = 0.5]{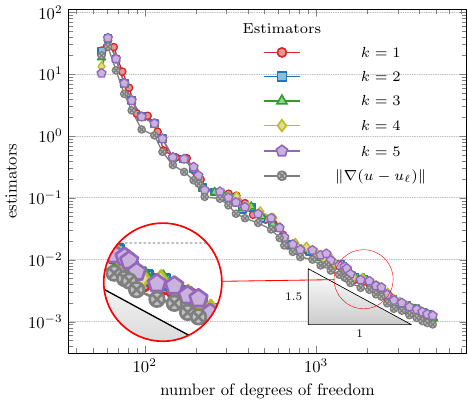}
   }
   \subfloat{
    \includegraphics[scale = 0.5]{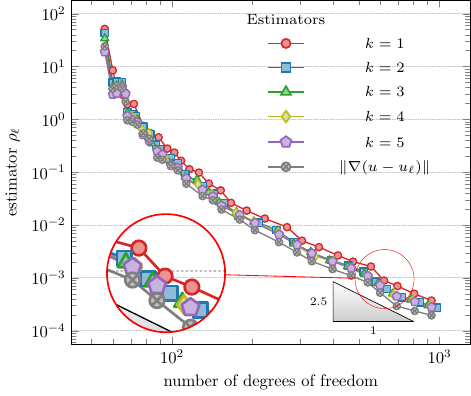}
   }
  }
   \caption{Comparison between $\eta_\ell$ computed with different patch sizes $k \in \set{1,2,3,4,5}$ and the exact error for $\theta = 0.4$, $p = 0$~(left) and $p = 1$~(right) in Example~\ref{subsec:squareex}.}
   \label{fig:sk}
 \end{figure}
In contrast to the global functional error estimator (see~\cite{Kurz2021,Freiszlinger2025}), the local functional error estimator $\eta_\ell$ does not exhibit a strong dependence on the patch size $k$ and already for $k = 1$ it provides an accurate approximation, which is desirable as the computational cost grows with $k$. 

\subsection{Example 2 (L-shaped domain, singular solution)} \label{subsec:lshapex}

We consider the Laplace--Dirichlet equation~\eqref{eq:Laplace} on a rotated L-shaped domain $\Omega$ (see Figure~\ref{fig:lmeshes}) with prescribed exact solution
\begin{equation*}
 u^\star(r,\varphi) \coloneqq r^{2/3} \sin(2\varphi/3)
\end{equation*}
given in polar coordinates $(r,\varphi)$.
Unlike Example~\ref{subsec:squareex}, the solution $u^\star$ exhibits a singularity at the reentrant corner and hence is not smooth in $\Omega$.
We start Algorithm~\ref{algo:adap} with an initial volume mesh $\TT_0$ consisting of $12$ triangles and $8$ boundary edges.
Figure~\ref{fig:lmeshes} shows the initial mesh and some adaptively generated meshes for $\theta = 0.4$ and $k = 3$.
\begin{figure}[!ht]
 \resizebox{\textwidth}{!}{
   \subfloat{
    \stackunder[10pt]{\includegraphics[scale = 0.5]{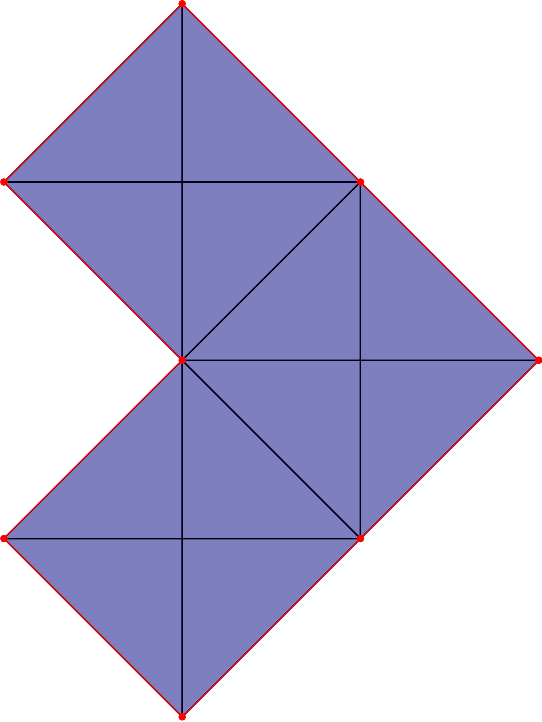}}{{$\# \TT_\ell^\Gamma = 8$, $\ell = 0$}}
   }
   \subfloat{
    \stackunder[10pt]{\includegraphics[scale = 0.5]{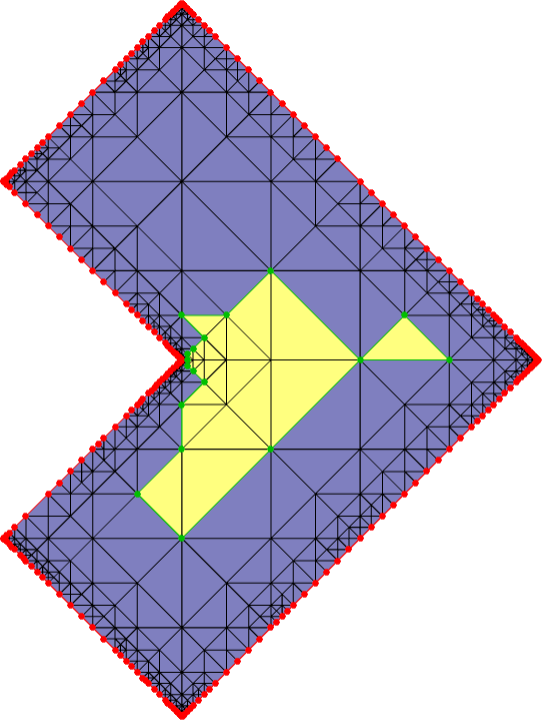}}{{ $\# \TT_\ell^\Gamma = 351$, $\ell = 50$}}
   }
   
   \subfloat{
   \stackunder[10pt]{\includegraphics[scale = 0.5]{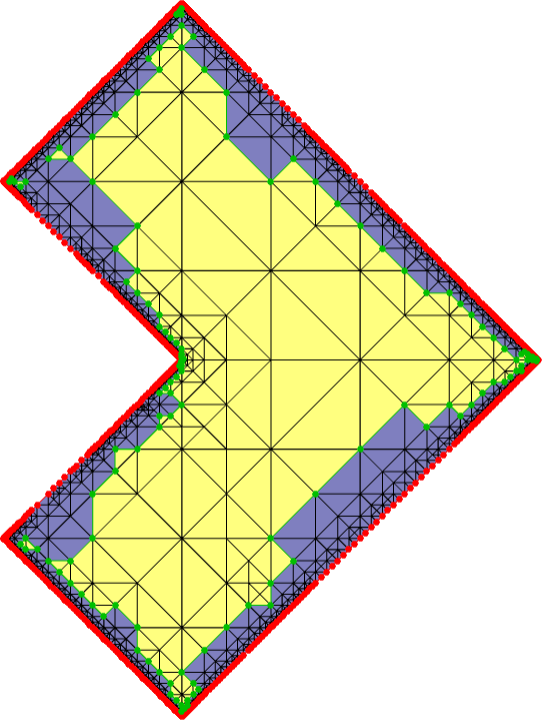}}{{ $\# \TT_\ell^\Gamma = 1133$, $\ell = 70$}}
   }
  }
   \caption{Meshes generated by Algorithm~\ref{algo:adap} in Example~\ref{subsec:lshapex} for $\theta = 0.4$ and $k = 3$. The triangles of the patches $\omega_{\ell,z}$ are depicted in \ul{blue}\setulcolor{yellow}, while the remaining triangles are indicated in \ul{yellow}\setulcolor{red}. The outer boundary $\Gamma$ is shown in \ul{red}\setulcolor{green} and the inner boundary of the union of the patches in \ul{green}\setulcolor{blue}.}
    \label{fig:lmeshes}
\end{figure}
We observe that the adaptive algorithm is able to resolve the singularity at the reentrant corner adequately. 
The total upper bound $\mu_\ell$ from~\eqref{eq:funcest} for $p = 0$~(left) and $p = 1$~(right) for patch size $k = 3$ and different adaptivity parameters $\theta \in \set{0.2,0.4,0.6,0.8,1}$ is shown in Figure~\ref{fig:lrates}.
\begin{figure}[!ht]
  \resizebox{\textwidth}{!}{
   \subfloat{
    \includegraphics[scale = 0.5]{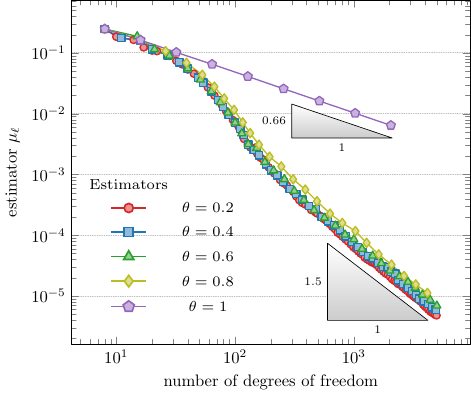}
   }
   \subfloat{
    \includegraphics[scale = 0.5]{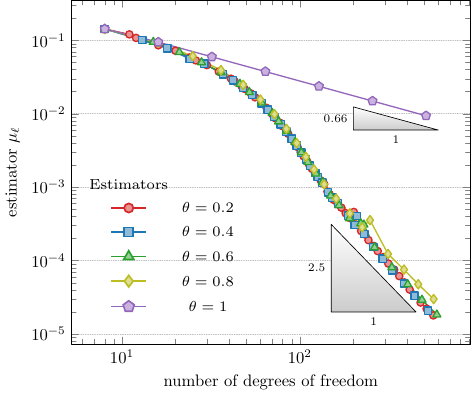}
   }
  }
   \caption{Convergence rates of the full error estimator $\mu_\ell$ from~\eqref{eq:funcest} generated by Algorithm~\ref{algo:adap} in Example~\ref{subsec:lshapex} for $k = 3$, different marking parameters $0 < \theta \leq 1$, $p = 0$~(left) and $p = 1$~(right).}
   \label{fig:lrates}
 \end{figure}
We observe that, for $p \in \set{0,1}$, the error estimator $\mu_\ell$ decays with the optimal convergence rate for $0 < \theta < 1$, while uniform refinement ($\theta = 1$) leads to suboptimal convergence rates.
A comparison between the error estimator $\eta_\ell$ from Section~\ref{sec:funcest}, the functional error estimator obtained by globally solving the auxiliary problem, the residual error estimator $\rho_\ell$ from~\eqref{eq:resest}, the Faermann error estimator, the $h-h/2$ error estimator, and the exact error $\norm{\nabla(u^\star - u_\ell^\star)}_\Omega$ for $p = 0$, $k = 3$, and $\theta = 0.4$ is shown in Figure~\ref{fig:lcomp}~(left).
The corresponding experimental reliability constants $\mathtt{estimator} / \norm{\nabla(u^\star - u_\ell^\star)}_\Omega$ are shown in Figure~\ref{fig:lcomp}~(right).
\begin{figure}[!ht]
  \resizebox{\textwidth}{!}{
   \subfloat{
    \includegraphics[scale = 0.5]{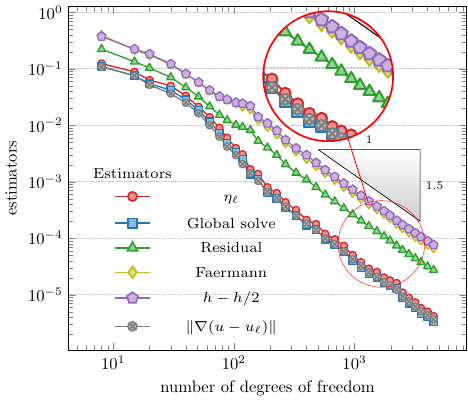}
   }
   \subfloat{
    \includegraphics[scale = 0.5]{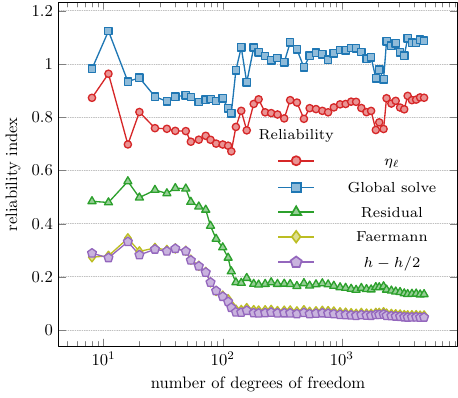}
   }
  }
   \caption{Comparison between $\eta_\ell$, the functional estimator obtained by a global solve, the residual error estimator $\rho_\ell$, the Faermann error estimator, the $h-h/2$ error estimator, and the exact error in terms of convergence rates~(left) and experimental reliability constants~(right) for $p = 0$, $k = 3$, and $\theta = 0.4$ in Example~\ref{subsec:lshapex}.}
   \label{fig:lcomp}
 \end{figure}
As in Example~\ref{subsec:squareex}, the functional error estimators exhibit excellent quality while the remaining estimators show a significant overestimation of the exact error $\norm{\nabla(u^\star - u_\ell^\star)}_\Omega$.
Lastly, Figure~\ref{fig:lk} shows the exact error as well as the error estimator $\eta_\ell$ for $\theta = 0.4$, different patch sizes $k \in \set{1,2,3,4,5}$, $p = 0$~(left) and $p = 1$~(right).
\begin{figure}[!ht]
  \resizebox{\textwidth}{!}{
   \subfloat{
    \includegraphics[scale = 0.5]{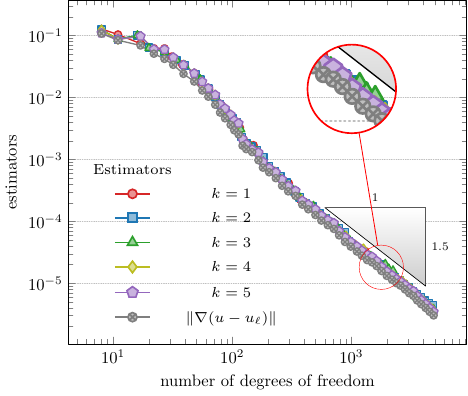}
   }
   \subfloat{
    \includegraphics[scale = 0.5]{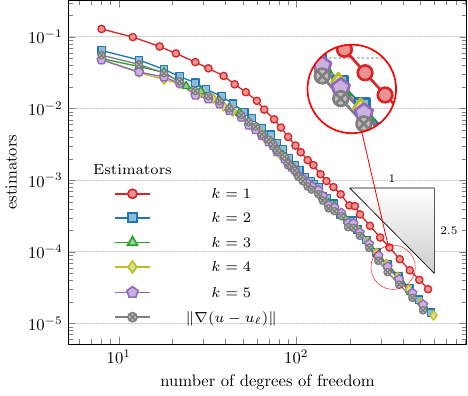}
   }
  }
   \caption{Comparison between $\eta_\ell$ computed with different patch sizes $k \in \set{1,2,3,4,5}$ and the exact error for $\theta = 0.4$, $p = 0$~(left) and $p = 1$~(right) in Example~\ref{subsec:lshapex}.}
   \label{fig:lk}
 \end{figure}
As in Example~\ref{subsec:squareex}, the local functional error estimator $\eta_\ell$ does not depend strongly on the patch size $k$.
However, in contrast to Example~\ref{subsec:squareex}, the error estimator $\eta_\ell$ for $p = 1$ and $k = 1$ is not as accurate as for larger patch sizes.
Still, already for $k = 2$ the error estimator $\eta_\ell$ provides an accurate approximation of the exact error, while the computational cost is still relatively low.

\printbibliography

\end{document}